\documentclass[final,leqno]{siamltex}
\usepackage{color}
\usepackage{amsfonts}
\usepackage{threeparttable}
\usepackage{amsmath}
\usepackage{graphicx}
\usepackage{multirow}
\usepackage{listings}
\usepackage{epstopdf}

\newtheorem{example}{Example}
\newtheorem{assumption}{Assumption}[section]

\title{Convolution quadratures based on block generalized Adams methods
\thanks{This work was supported by National Natural Science Foundation of China (No. 11901133) and Training Program of Guizhou University (No. GDPY[2020]39).}}

\author{Ling Liu\thanks{School of Mathematics and Statistics, Guizhou University, Guiyang, Guizhou
550025, P.R. China.}
        \and Junjie Ma\thanks{School of Mathematics and Statistics, Guizhou University, Guiyang, Guizhou
550025, P.R. China.({\tt jjma@gzu.edu.cn}).}
}

\begin{document}

\maketitle

\begin{abstract}
This paper studies a family of convolution quadratures, a numerical technique for efficient evaluation of convolution integrals.
We employ the block generalized Adams method to discretize the underlying initial value problem,
departing from the well-established approaches that rely on linear multistep formulas or Runge-Kutta methods.
The convergence order of the proposed convolution quadrature
can be dynamically controlled without requiring grid point adjustments, enhancing flexibility.
Through strategic selection of the local interpolation polynomial and block size, the method achieves high-order convergence for calculation of convolution integrals with hyperbolic kernels.
We provide a rigorous convergence analysis for the proposed convolution quadrature and numerically validate our theoretical findings for various convolution integrals.
\end{abstract}



\pagestyle{myheadings}
\thispagestyle{plain}

\section{Introduction}

Convolution integrals (CIs) have a number of
applications in computational practice, such as
time-domain boundary integral equations \cite{banjai2022integral},
Volterra integral equations of convolution-type \cite{brunner2017volterra},
fractional differential equations \cite{banjai2019efficient,jin2017correction,shi2023high} and
Maxwell's equations \cite{wang2008a}.
In this paper, we investigate the numerical evaluation of the following CI
\begin{align}
\label{CIi}
 \int_{0}^tk(s)g(t-s)ds~~\mathrm{with}~~t\in [0,T],
\end{align}
where $g$ is a given smooth function.
When the kernel function $k$ itself is known previously,
many efficient approaches can be used to
compute  CI \eqref{CIi},
for example,
the convolution spline method \cite{davis2013convolution} and
spectral method based on orthogonal polynomial convolution matrices
\cite{xu2017a,xu2018spectral}.
However, if the Laplace transform $K$ of the
kernel function $k$ is known,
then the convolution quadrature (CQ),
which was proposed by Lubich (see
\cite{lubich1988convolution1,lubich1988convolution2,lubich2004convolution}),
becomes an attractive choice due to its suitability for stable integration over long time intervals.

For clarity, we introduce the following notation to represent CI \eqref{CIi} (see also \cite{banjai2011an}),
\begin{align}
\label{CI}
(K(\partial_t)g)(t) := \int_{0}^tk(s)g(t-s)ds~~\mathrm{with}~~t\in [0,T].
\end{align}
Here we have $(\partial_t^{-1}g)(t)=\int_0^tg(s)ds$
and $(\partial_tg)(t) = g'(t)+g(0)\delta(t)$
with $\delta$ denoting Dirac delta function.
Noting that
\begin{align}
\label{LapK}
k(s)=\frac{1}{2\pi \mathrm{i}}\int_{\Gamma}K(\lambda)\mathrm{e}^{\lambda s}d\lambda~~\mathrm{with}~~s>0,
\end{align}
where $\Gamma$  is a carefully chosen contour that
falls in the analytic region
of $K,$
we can compute
\begin{align}
\label{CI1}
(K(\partial_t)g)(t)=\frac{1}{2\pi \mathrm{i}}\int_{\Gamma}K(\lambda)
\int_{0}^t\mathrm{e}^{\lambda s}g(t-s)dsd\lambda.
\end{align}
Letting $ y(t)=\int_{0}^t\mathrm{e}^{\lambda s}g(t-s)ds,$
we obtain the  initial value problem (IVP)
for the ordinary differential equation (ODE):
\begin{align}
\label{IVP}
\left\{
  \begin{array}{l}
   y'(t)=\lambda y(t)+g(t)~~\mathrm{with}~~t>0,  \\
   y(0)=0.
  \end{array}
\right.
\end{align}
The calculation of CI \eqref{CI} then consists of two steps: approximating IVP \eqref{IVP} using time-stepping methods, and computing
contour integrals.

The literature boasts extensive research on two primary types of ODE solvers for IVP \eqref{IVP}.
Lubich's pioneering work in
\cite{lubich1988convolution1,lubich1988convolution2}
established the foundation for the linear multistep convolution quadrature (LMCQ) by applying linear multistep formulas for discretization of IVP \eqref{IVP}.
Convergence analysis revealed that LMCQ achieves the same convergence order as the underlying linear multistep formula, but only after adding correction terms.
LMCQ is implemented on a uniform grid and offers the advantage of
adjusting its convergence order without modifying the grid points.
This implies changing the convergence order of LMCQ does not need to re-evaluate $g,$ since evaluating the function $g$ depends solely on the grid points.
This advantage extends to adaptive algorithms as well. When doubling the grid points, LMCQ can efficiently reuse the previously computed values of $g.$
These combined advantages make LMCQ a compelling choice for a wide range of time-space problems.
These problems often involve computationally expensive spatial discretizations, such as those encountered with time-fractional differential operators
(see \cite{cuesta2006convolution,jin2017correction}).
However, $A-$stable multistep methods are limited by Dalquist's order barrier to a maximum convergence order of $2.$
Due to this limitation, LMCQs are typically constructed using lower-order formulas,
like the second-order backward difference (BDF2) \cite{cuesta2006convolution}
or the trapezoid rule (TR) \cite{banjai2023generalized,eruslu2020polynomially},
to guarantee $A-$stability when applied to hyperbolic problems.
In contrast, Runge-Kutta methods can achieve $A-$stability for any desired order by strategically selecting non-uniform stage nodes, such as Gauss-Legendre, Radau or Lobatto points \cite{banjai2011an,banjai2022runge}.
Therefore, CQ based on Runge-Kutta method (RKCQ)
can provide arbitrarily high-order and stable
approximations to CI \eqref{CI} with a hyperbolic kernel (see \cite{banjai2011an}).
Benefitting from its excellent stability, high-order RKCQ offers a broader range of applications compared to LMCQ.
This is particularly advantageous for numerical studies of the wave equation
(see \cite{banjai2022integral,melenk2021on,nick2022time,rieder2022time}).
In addition, recent studies also demonstrate numerous modifications to the original CQs,
including fast implementation techniques \cite{banjai2010multistep,banjai2014fast,fischer2019fast},
variable time stepping strategies
\cite{banjai2023generalized,lopez2013generalized,lopez2016generalized}
and efficient computation of highly oscillatory integrals \cite{ren2024a}.

The aim of this paper is to develop
a family of CQs that combines the strengths of LMCQs and RKCQs:
It is implemented on the uniform grid and offers the flexibility to adjust  the convergence order without requiring additional grid points as LMCQ does;
The proposed method can achieve high-order and stable calculation of CI \eqref{CI} with a hyperbolic kernel as RKCQ does.
This is achieved by leveraging the block generalized Adams method (BGA) as the foundation for the underlying ODE solvers,
which draws inspiration from the application of the boundary value method in the numerical solutions of ODEs (see \cite{brugnano1998solving,iavernaro1998solving,iavernaro1999block}).
Henceforth,
we suppose that $K(\lambda)$ is analytic in the half-plane
\begin{align*}
\mathbb{C}_{\sigma}:=\{z\in \mathbb{C}: \mathrm{Re}(z)\geq \sigma\}
~~\mathrm{with}~~\sigma\in \mathbb{R},
\end{align*}
and  bounded as $|K(\lambda)|\leq M|\lambda|^{\mu}$
within $\mathbb{C}_{\sigma}$ for a positive constant $M$ and
a real number $\mu.$

The remainder of the paper is structured as follows.
Section 2 introduces  BGA for IVP \eqref{IVP}.
Building upon BGA, Section 3 develops the BGA-based convolution quadrature (BGACQ). We then analyze BGACQ's convergence properties under hypotheses of regularity.
Section 4 presents numerical experiments to verify the proposed quadrature rule's convergence and stability.

\section{BGA for IVP \eqref{IVP}}\label{Sec2}

This section introduces the ODE solver, BGA,  for IVP \eqref{IVP} and explores several BGAs  which can be employ in the formulation of CQ for CI with a hyperbolic  kernel.

To begin with, we define the uniform coarse grid, denoted by $\mathcal{X}_N,$
\begin{align*}
\mathcal{X}_N:=\{\tau_n:=nh,~n=0,1,\cdots,N\},
\end{align*}
where $h=T/N$ represents the grid step size.
Within each subinterval $[\tau_n,\tau_{n+1}],$
we introduce a uniform fine grid that plays a role analogous to the stages in a Runge-Kutta method,
\begin{align*}
\mathbf{X}_m^n:=\{t^n_j:=\tau_n+jh/m,~j=0,1,\cdots,m\}.
\end{align*}
It can be easily examined that $t^n_0=\tau_{n}$ and $t^n_m=\tau_{n+1}.$
For $j=0,1,\cdots,m-1,$ integration of both sides of
IVP \eqref{IVP} from $t^n_{j}$ to $t^n_{j+1}$
gives
\begin{align}
\label{IVP2}
y(t_{j+1}^n)-y(t_{j}^n)=\int_{t_{j}^n}^{t_{j+1}^n}
(\lambda y(s)+g(s))ds=\int_{t_{j}^n}^{t_{j+1}^n}
f(s)ds,
\end{align}
where $f(s):=\lambda y(s)+g(s).$

Next, define the piecewise polynomial space by
\begin{align*}
S_{m}(\mathbf{X}_m^n):=
\{v\in C[\tau_n,\tau_{n+1}]:v|_{[t_j^n,t_{j+1}^n]}\in \pi_{m},j=0,1,\cdots,m-1\},
\end{align*}
where $C[\tau_n,\tau_{n+1}]$ comprises continuous functions over the
interval $[\tau_n,\tau_{n+1}],$
and $v|_{[t_j^n,t_{j+1}^n]}\in \pi_{m}$ indicates $v$
is a polynomial with its degree not exceeding $m$
over the subinterval $[t_j^n,t_{j+1}^n].$
Given any nonnegative integers $k_1$ and $k_2,$
we further define the local fundamental functions
\begin{align*}
\phi_i^{k_1,k_2}(v)=
\prod_{l=-k_1,l\neq i}^{k_2+1}\frac{v-l}{i-l}~~\mathrm{with}~~i=-k_1,\cdots,k_2+1.
\end{align*}
Then we denote $\mathcal{P}_m^n$ as an interpolation operator acting on functions in $C[\tau_n,\tau_{n+1}],$
mapping them to the space of $S_{m}(\mathbf{X}_m^n).$
When $s\in [t_j^n,t_{j+1}^n]$ with $j=k_1,k_1+1,\cdots,m-k_2-1,$
we define $\mathcal{P}_m^n[f]$ as
\begin{align*}
 \mathcal{P}_m^n[f](s)= \mathcal{P}_m^n[f](t_{j}^n+vh/m):=\sum_{i=-k_1}^{k_2+1}
f(t_{j+i}^n)\phi_i^{k_1,k_2}(v), ~~v\in [0,1].
\end{align*}
To ensure the solvability of the resulting discretization, we employ an ``additional method" over the initial interval
$[t_0^n,t_{k_1}^n]$
and the final interval $[t^n_{m-k_2},t^n_m],$
similar to the approach used in block boundary value methods
(see \cite[pp.283]{brugnano1998solving}).
When $s\in [t_0^n,t_{k_1}^n],$ we define $\mathcal{P}_m^n[f]$ as
\begin{align*}
 \mathcal{P}_m^n[f](s)=\mathcal{P}_m^n[f](t_{k_1}^n+vh/m):=
\sum_{i=-k_1}^{k_2+1}
f(t_{k_1+i}^n)\phi_i^{k_1,k_2}(v),
\end{align*}
where $v\in [-k_1,0].$
When $s\in [t^n_{m-k_2},t^n_m],$ we define $\mathcal{P}_m^n[f]$ as
\begin{align*}
 \mathcal{P}_m^n[f](s)=\mathcal{P}_m^n[f](t_{m-k_2-1}^n+vh/m):=
\sum_{i=-k_1}^{k_2+1}
f(t_{m-k_2-1+i}^n)\phi_i^{k_1,k_2}(v),
\end{align*}
where $v\in [1,k_2+1].$
For $i=-k_1,\cdots,k_2+1,$
define the notations
\begin{align*}
\alpha_i:=&\frac{1}{m}\int_0^1\phi_i^{k_1,k_2}(v)dv,
\\
\alpha_i^{(j)}:=&\frac{1}{m}\left\{
                   \begin{array}{ll}
                     \int_{j-k_1}^{j-k_1+1}\phi_i^{k_1,k_2}(v)dv,&
j=0,\cdots,k_1-1, \\
                     \int_{j-m+k_2+1}^{j-m+k_2+2}\phi_i^{k_1,k_2}(v)dv,&
j=m-k_2,\cdots,m-1.
                   \end{array}
                 \right.
\end{align*}
This yields the quadrature rule $\mathrm{Q}_j^n[f]:= \int_{t_j^n}^{t_{j+1}^n}\mathcal{P}_m^n[f](s)ds$ for the integral
$\int_{t_j^n}^{t_{j+1}^n}f(s)ds,$
that is,
\begin{align}
\label{Qerr}
\mathrm{Q}_j^n[f]=h\sum_{i=-k_1}^{k_2+1}
\left\{
  \begin{array}{ll}
f(t_{k_1+i}^n)
\alpha_i^{(j)} ,&j=0,\cdots,k_1-1,\\
f(t_{j+i}^n)
\alpha_i,&j=k_1,\cdots,m-k_2-1,\\
f(t_{m-k_2-1+i}^n)
\alpha_i^{(j)},&j=m-k_2,\cdots,m-1.
  \end{array}
\right.
\end{align}

Before delving into the discretization of IVP \eqref{IVP}, let's examine the quadrature error $q^n_j$ defined by
\begin{align*}
q^n_j:=\int_{t_j^n}^{t_{j+1}^n}(\mathcal{I}-\mathcal{P}_m^n)[f](s)ds,
\end{align*}
where $\mathcal{I}$ denotes the  identity transformation.
In fact, the Peano kernel theorem \cite[pp. 43]{brunner2004collocation} allows us to quantify the difference between the exact function $f$ and its approximate piecewise polynomial $\mathcal{P}_m^n[f]$ over the subinterval $[t^n_j,t^n_{j+1}].$
Denote
\begin{align*}
\mathcal{E}^n_{j}(s):= \int_{-k_1}^{k_2+1}\kappa(s,v)
   f^{(k_1+k_2+2)}(t^n_{j}+vh/m)dv,  j=k_1,\cdots,m-k_2-1,
\end{align*}
where $f^{(k_1+k_2+2)}(t^n_{j}+vh/m)$ represents the $(k_1+k_2+2)-$th order derivative of
$f(t)$ at $t^n_{j}+vh/m,$ and the kernel $\kappa(s,v)$ is defined by
\begin{align*}
\kappa(s,v):=\frac{1}{(p-1)!}\left((s-v)_+^{k_1+k_2+1}
-\sum_{i=-k_1}^{k_2+1}\phi_i^{k_1,k_2}(s)(i-v)_+^{k_1+k_2+1}\right)
\end{align*}
with $(s-v)_+:=0$ for $s<v$ and
$(s-v)_+:=s-v$ for $s\geq v.$
In the case of $j = 0,\cdots,k_1-1,$
letting $t=t^n_{k_1}+sh/m,$
we   represent the remainder $(\mathcal{I}-\mathcal{P}_m^n)[f](t)$ in the form of
\begin{align*}
(\mathcal{I}-\mathcal{P}_m^n)[f](t^n_{k_1}+sh/m)
=\frac{h^{k_1+k_2+2}}{m^{k_1+k_2+2}}
\mathcal{E}^n_{k_1}(s), ~~ s\in [j-k_1,j-k_1+1].
\end{align*}
When $t\in [t^n_j,t^n_{j+1}]$ with $j=k_1,\cdots,m-k_2-1,$
letting $t=t^n_{j}+sh/m,$
we  represent $(\mathcal{I}-\mathcal{P}_m^n)[f](t)$ by
\begin{align*}
(\mathcal{I}-\mathcal{P}_m^n)[f](t^n_{j}+sh/m)
=\frac{h^{k_1+k_2+2}}{m^{k_1+k_2+2}}
\mathcal{E}^n_{j}(s), ~~ s\in [0,1].
\end{align*}
In the case of $j=m-k_2,\cdots,m-1,$
letting $t=t^n_{m-k_2-1}+sh/m,$
for $s\in [j-m+k_2+1,j-m+k_2+2],$
the approximation error is represented by
\begin{align*}
(\mathcal{I}-\mathcal{P}_m^n)[f](t^n_{m-k_2-1}+sh/m)
=\frac{h^{k_1+k_2+2}}{m^{k_1+k_2+2}}
\mathcal{E}^n_{m-k_2-1}(s).
\end{align*}
Consequently, the quadrature error $q^n_j$
is expressed as
\begin{align}
\label{Quaderr}
q^n_j =
\frac{h^{k_1+k_2+3}}{m^{k_1+k_2+3}}
 \left\{
  \begin{array}{ll}
   \int_{j-k_1}^{j-k_1+1}\mathcal{E}^n_{j}(s)ds, &j=0,\cdots,k_1-1, \\
   \int_{0}^{1}\mathcal{E}^n_{j}(s)ds, &j=k_1,\cdots,m-k_2-1, \\
    \int_{j-m+k_2+1}^{j-m+k_2+2}\mathcal{E}^n_{j}(s)ds, &j=m-k_2,\cdots,m-1.
  \end{array}
\right.
\end{align}

We let the curve $\Gamma$ to be $\sigma+\mathrm{i}\mathbb{R}$ and
introduce the assumption on the function $g:$
\begin{align*}
g(0) = g'(0)=\cdots=g^{(k_1+k_2+1)}(0)=0.
\end{align*}
This assumption will hold throughout the following
discussion unless explicitly stated otherwise.
Noting that
\begin{align*}
y(t)=\int_0^t\mathrm{e}^{\lambda s}g(t-s)ds~~\mathrm{and}~~
f(t)=\lambda y(t)+g(t),
\end{align*}
we obtain through a straightforward calculation
\begin{align}\label{fasy1}
f^{(k_1+k_2+2)}(t)=\lambda\int_0^tg^{(k_1+k_2+2)}(t-s)\mathrm{e}^{\lambda s}ds+ g^{(k_1+k_2+2)}(t).
\end{align}
Since $\int_0^tg^{(k_1+k_2+2)}(t-s)\mathrm{e}^{\lambda s}ds$
becomes highly oscillatory  when $\lambda\in \Gamma$ and $|\lambda|>1,$
we derive according to  van der Corput lemma \cite[pp.332]{stein1993harm},
\begin{align*}
\|f^{(k_1+k_2+2)}\|\leq C~~\mathrm{for~~any}~~
\lambda\in \Gamma,
\end{align*}
where $\|\cdot\|$ denotes the
maximum norm.
Throughout the following analysis, $C$ represents a generic positive constant that does not depend on $\lambda$ or $h.$
It follows from Eq. \eqref{Quaderr} that
\begin{align*}
|q_j^n|\leq Ch^{k_1+k_2+3}
~~\mathrm{for~~any}~~\lambda\in\Gamma.
\end{align*}

\begin{lemma}\label{AdaErr}
Assume the function $g$ in CI \eqref{CI} possesses a continuous derivative of order at least $k_1+k_2+2$
over $[0,T],$
and satisfies the condition
\begin{align*}
g(0) = g'(0)=\cdots=g^{(k_1+k_2+1)}(0)=0.
\end{align*}
Then the  error $q^n_j$
generated by the quadrature rule defined in Eq. \eqref{Qerr} satisfies
\begin{align*}
|q^n_j|\leq Ch^{k_1+k_2+3},~~ as~h \rightarrow 0.
\end{align*}
\end{lemma}

We now turn our attention to developing   BGA  for IVP \eqref{IVP}.
Denote the matrices $\tilde{\mathbf{A}}$ and $\tilde{\mathbf{L}}$ of dimension $m\times (m+1)$
\begin{align*}
\tilde{\mathbf{A}}
:=&
\left(
  \begin{array}{c|ccccccc}
    \alpha^{(0)}_{-k_1} & \alpha^{(0)}_{-k_1+1} & \cdots & \alpha^{(0)}_{-k_2+1} &  &  &  &  \\
    \vdots &\vdots &   & \vdots &   &  &   &   \\
    \alpha^{(k_1-1)}_{-k_1} & \alpha^{(k_1-1)}_{-k_1+1} & \cdots & \alpha^{(k_1-1)}_{-k_2+1} &  &  &  &   \\
    \alpha_{-k_1} &  \alpha_{-k_1+1} & \cdots &  \alpha_{k_2+1} &   &   &  &  \\
     & \alpha_{-k_1} &  \cdots & \alpha_{k_2} &\alpha_{k_2+1}  &   &  &   \\
      &   & \ddots  & \ddots  &  \ddots & \ddots  &  &   \\
    &  &  & \ddots  &  \ddots &\ddots   & \ddots  &   \\
      &   &   &   &   & \alpha_{-k_1} &  \cdots & \alpha_{k_2+1} \\
      &   &   &   &   & \alpha_{-k_1}^{(m-k_2)} &  \cdots & \alpha_{k_2+1}^{(m-k_2)} \\
 &   &   &   &   & \vdots &    & \vdots\\
    &   &   &   &   & \alpha_{-k_1}^{(m-1)} &  \cdots & \alpha_{k_2+1}^{(m-1)} \\
  \end{array}
\right)
\\
:=&[\mathbf{a}~|~\mathbf{A}]
\end{align*}
and
\begin{align*}
\tilde{\mathbf{L}}:=
\left(
  \begin{array}{c|cccccc}
 -1& 1  &   &   &   &   &  \\
 & -1  &  1 &   &   &   &  \\
 &   &  \ddots & \ddots  &   &    & \\
 &   &   &   & \ddots  &  \ddots&  \\
 &   &   &   &   &  -1&1 \\
  \end{array}
\right):=[\mathbf{l}~|~\mathbf{L}].
\end{align*}
BGA with block size $m$  discretizes Eq. \eqref{IVP2} into
\begin{align*}
y_m^{n-1}\mathbf{l}+\mathbf{L}\mathbf{Y}_{n}=
h\lambda y_m^{n-1}\mathbf{a}+h\lambda \mathbf{A}\mathbf{Y}_{n}
+hg_m^{n-1}\mathbf{a}+h\mathbf{A}\mathbf{G}_{n},
~n=0,\cdots,N-1,
\end{align*}
or equivalently,
\begin{align}
\label{IVPnum}
(\mathbf{L}-h\lambda \mathbf{A})\mathbf{Y}_{n}
=(h\lambda \mathbf{a}-\mathbf{l})y_m^{n-1}+hg_m^{n-1}\mathbf{a}
+h\mathbf{A}\mathbf{G}_{n},~n=0,\cdots,N-1,
\end{align}
where $\mathbf{G}_{n}:=(g(t_1^{n})~~\cdots~~g(t_m^{n}))^T,$ $\mathbf{Y}_{n}:=(y^{n}_1~~\cdots~~y^{n}_m)^T$ denotes
the approximation to
$\mathbf{Y}_{n}^{ref}:=(y(t_1^{n})~~\cdots~~y(t_m^{n}))^T,$
 and $y_m^{-1}=0, ~g_m^{-1}=g(0).$
Noting that
\begin{align*}
y_m^n=  \mathbf{e}_m^{T}(\mathbf{L}-h\lambda \mathbf{A})^{-1}
(h\lambda \mathbf{a}-\mathbf{l})y_m^{n-1}
+\mathbf{e}_m^{T}(\mathbf{L}-h\lambda \mathbf{A})^{-1}
(hg_m^{n-1}\mathbf{a}
+h\mathbf{A}\mathbf{G}_{n}),
\end{align*}
where $\mathbf{e}_i$ represents a vector of dimension $m\times 1$ with all entries zero except for the $i-$th entry being $1,$
we define the stability function, $\mathrm{R}_{m}(z),$ for BGA by
\begin{align*}
\mathrm{R}_{m}(z)=\mathbf{e}_m^{T}(\mathbf{L}-z \mathbf{A})^{-1}
(z \mathbf{a}-\mathbf{l}).
\end{align*}
Let
\begin{align*}
\mathbb{C}^{-}:=\{z\in \mathbb{C}:\mathrm{Re}(z)\leq 0\},
~~\mathbb{C}^{+}:=\mathbb{C}\setminus \mathbb{C}^{-}.
\end{align*}
To ensure the resulting CQ can handle CI \eqref{CI} with a hyperbolic kernel,
we introduce the following assumption regarding the ODE solver.
\begin{assumption}
\label{ASS}
Suppose the stability function of BGA is defined as $\mathrm{R}_{m}(z)=\mathbf{e}_m^{T}(\mathbf{L}-z \mathbf{A})^{-1}
(z \mathbf{a}-\mathbf{l}),$
and the following statement holds true:
\begin{itemize}
  \item $\mathbf{A}$ is invertible, and all eigenvalues of $\mathbf{A}^{-1}\mathbf{L}$
    are positive;
  \item $|\mathrm{R}_{m}(\mathrm{i}\omega)|<1,~~\omega\neq 0;$
\item $|\mathrm{R}_{m}(\infty)|=|\mathbf{e}_m^T\mathbf{A}^{-1}\mathbf{a}|<1.$
\end{itemize}
\end{assumption}

The first condition implies that $\mathbf{L}-z\mathbf{A}$ is invertible for all
$z\in \mathbb{C}^-.$
Consequently, $\mathrm{R}_{m}(z)$ is analytic within this half-plane.
Since $\mathrm{R}_{m}(0)=1,$
the second condition guarantees that   $|\mathrm{R}_{m}(z)|\leq 1$
holds for all  $z\in \mathbb{C}^-$ according to
the maximum-modulus principle.
Therefore, these first two conditions together ensure that the corresponding BGA is $A-$stable.
From a numerical implementation standpoint, the second condition can be relaxed to
$|\mathrm{R}_{m}(\mathrm{i}\omega)|\leq 1$ for $\omega\in \mathbb{R}.$
The modified second condition and the third condition play a twofold role:
they simplify the convergence order proof and ensure the corresponding ODE solver can solve stiff problems more efficiently compared to solvers where$|\mathrm{R}_{m}(\infty)|=1.$
Table \ref{Eig} summarizes  the
approximate eigenvalues of
$\mathbf{A}^{-1}\mathbf{L}$ for several BGAs with $m^*=k_1+k_2+2.$
The numerical experiments suggest that all eigenvalues of
 $\mathbf{A}^{-1}\mathbf{L}$
seem to have positive real parts in these cases.
In Figure \ref{SF1}, we plot the stability boundaries for BGAs.
These boundaries represent the curves where $|\mathrm{R}_{m}(z)|=1.$
Because $\mathrm{R}_{m}(z)$
is analytic outside these boundaries,
the maximum-modulus principle guarantees that its modulus remains less than or equal to $1$ within that region.
Based on the numerical results, it seems that
 these schemes are $A-$stable.
Table \ref{Infinity} presents the limits, $|\mathrm{R}_{m}(\infty)|,$
for these $A-$stable methods.
As the table shows, all limits are strictly less than $1.$
Numerical experiments also confirm that this limit approaches $0$ as
$m$ increases.
While arbitrary choices of $k_1,k_2$ and $m$
don't generally guarantee that BGA satisfies Assumption \ref{ASS},
we can identify a minimum block size numerically to ensure this property holds.
We list the minimum block size $m^*$ for several BGAs satisfying
Assumption \ref{ASS} in Table \ref{minisize},
and will not explore the broader range of BGAs that might also satisfy
this assumption.

\begin{table}
\tabcolsep 0pt \caption{Numerical approximations to eigenvalues of $\mathbf{A}^{-1}\mathbf{L}$.} \label{Eig}
\vspace*{-10pt}
\begin{center}
\def\temptablewidth{1\textwidth}
{\rule{\temptablewidth}{0.8 pt}}
\begin{tabular*}{\temptablewidth}{@{\extracolsep{\fill}}ccc}
\hline
                & $m=m^*$ &$m=12$
                \\ \hline
&   & 19 $\pm$ 2.1i, 18 $\pm$ 5.9i,\\
     $(k_1,k_2)=(0,1)$&    5.5, ~3.6 $\pm$ 2.8i &   11 $\pm$ 12i, 12$\pm$ 11i, \\
         &    & 16 $\pm$ 8.7i, 13 $\pm$ 11i
                \\ \hline
&                &  5.7 $\pm$ 14i, ~7.6 $\pm$ 14i,
\\
  $(k_1,k_2)=(0,2)$    &
                5.7 $\pm$ 1.1i, ~3.3 $\pm$ 4.9i, &10 $\pm$ 12i, ~12 $\pm$ 8.4i,\\
&&               13 $\pm$ 4.4i, ~13 $\pm$ 1.0i
                \\ \hline
&&  8.2 $\pm$ 17i, ~10 $\pm$ 15i,\\
           $(k_1,k_2)=(1,2)$      & 3.7 $\pm$ 7.0i, ~7.8, ~6.3 $\pm$ 2.9i
                 &13 $\pm$ 13i, ~15 $\pm$ 8.7i,\\
       &&     15 $\pm$ 5.1i, ~16 $\pm$ 1.8i
                \\ \hline
       \end{tabular*}
       {\rule{\temptablewidth}{1 pt}}
       \end{center}
       \end{table}

\begin{figure}
\begin{center}
\includegraphics[width=12cm,height=6cm]{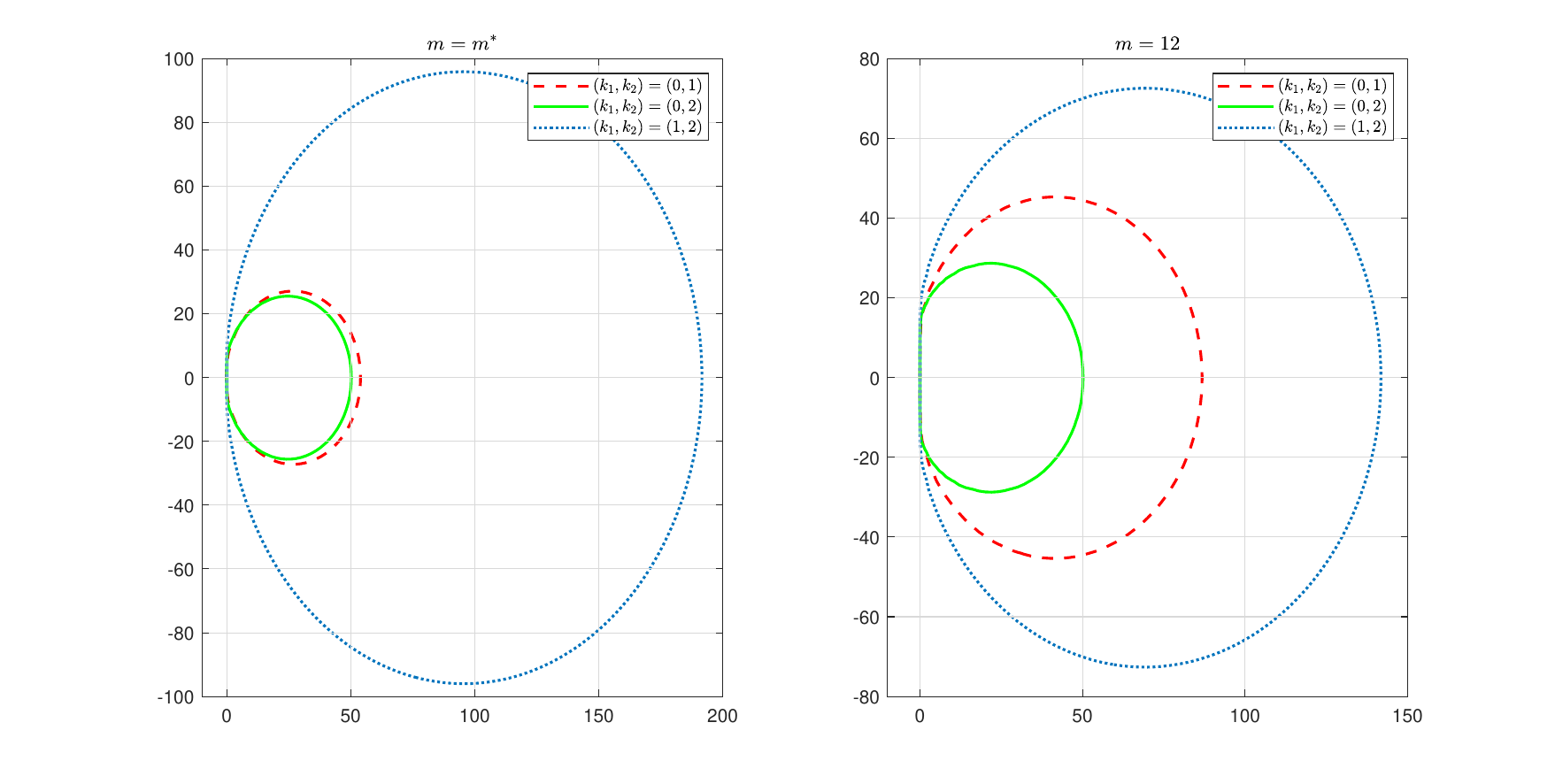}
\caption{Boundaries of the stability regions
 for  BGAs with $m=m^*$ (left) and $m=12$ (right).}\label{SF1}
\end{center}
\end{figure}

\begin{table}
\tabcolsep 0pt \caption{$|\mathrm{R}_{m}(\infty)|$ for various $A-$stable BGAs. }\label{Infinity}
\vspace*{-10pt}
\begin{center}
\def\temptablewidth{1\textwidth}
{\rule{\temptablewidth}{0.8pt}}
\begin{tabular*}{\temptablewidth}{@{\extracolsep{\fill}}ccccc}
\hline

        $m$        & $ m^*$   &$ 12$  & $ 24$   &$48$
                \\ \hline
                $(k_1,k_2)=(0,1)$
                &$5.9\times 10^{-1}$   &$4.5\times 10^{-3}$   &$6.9\times 10^{-6}$   &$1.6\times 10^{-11}$
\\
                $(k_1,k_2)=(0,2)$
                 &$4.3\times 10^{-2}$   &$4.4\times 10^{-4}$    &$1.4\times 10^{-8}$   &$1.5\times 10^{-17}$
                \\
                $(k_1,k_2)=(1,2)$
                &$7.4\times 10^{-1}$   &$7.6\times 10^{-2}$   &$1.5\times 10^{-3}$  &$6.3\times 10^{-7}$
                 \\ \hline
\end{tabular*}
{\rule{\temptablewidth}{1pt}}
\end{center}
\end{table}

\begin{table}
\tabcolsep 0pt \caption{Minimum block size $m^*$ for BGA satisfying Assumption \ref{ASS}. }\label{minisize}
\vspace*{-10pt}
\begin{center}
\def\temptablewidth{1\textwidth}
{\rule{\temptablewidth}{0.8pt}}
\begin{tabular*}{\temptablewidth}{@{\extracolsep{\fill}}ccccccccc}
\hline

        $(k_1,k_2)$    & $(0,1)$   &$(0,2)$   & $(1,2)$    &$(1,3)$
&$(2,3)$ &$(2,4)$ &$(3,4)$ &$(3,5)$
\\
\hline
$ m^*$ &3&4&5&6&7&10&13&13
                \\ \hline
\end{tabular*}
{\rule{\temptablewidth}{1pt}}
\end{center}
\end{table}

In conclusion, we demonstrate that BGA can be employed to construct a family of
$A-$stable ODE solvers suitable for designing CQs.
These BGA-based schemes leverage a uniform grid, eliminating the need for special points required by RK and allowing for adjustment of convergence order without grid point changes.
While numerical evidence suggests BGA itself is not $L-$stable,
it exhibits a significantly smaller magnitude of the stability function at infinity
compared to conservative schemes.
This characteristic, similar to $L-$stable RKCQs, makes CQ based on BGA less susceptible to issues like spurious oscillations, as discussed in \cite{banjai2010multistep}.
Compared to RK, increasing the block size of BGA offers a significant advantage due to the Toeplitz-like structure of $\mathbf{A}.$
This allows us to easily capture global information using relatively large time steps, as demonstrated in \cite[pp.289]{brugnano1998solving}.

\section{BGACQ for CI \eqref{CI}}

This section explores the application of BGA to construct a family of  CQs for computing CI \eqref{CI} with a hyperbolic kernel.
We will also analyze the convergence property of this method.

Multiplying both sides of Eq. \eqref{IVPnum} by $\zeta^{n}$
and doing summation over $n$ from $0$ to $\infty$ give
\begin{align}\label{EqwithZeta}
(\mathbf{L}-h\lambda \mathbf{A})\mathbf{Y}(\zeta)
=\zeta(h\lambda \mathbf{a}-\mathbf{l})\mathbf{e}_m^T\mathbf{Y}(\zeta)
+h(\mathbf{A}+\zeta \mathbf{a}\mathbf{e}_m^T)\mathbf{G}(\zeta),
\end{align}
where
$\mathbf{Y}(\zeta)=\sum_{n=0}^{\infty}\mathbf{Y}_n\zeta^n,$
$\mathbf{G}(\zeta)=\sum_{n=0}^{\infty}\mathbf{G}_n\zeta^n.$
A direct calculation leads to
\begin{align}
\label{YGeq}
((\mathbf{L}+\zeta \mathbf{l}\mathbf{e}_m^T)
-h\lambda(\mathbf{A}+\zeta\mathbf{a}\mathbf{e}_m^T))\mathbf{Y}(\zeta) = h(\mathbf{A}+\zeta \mathbf{a}\mathbf{e}_m^T)\mathbf{G}(\zeta).
\end{align}
Since the spectral radius of $\mathbf{A}^{-1}\mathbf{a}\mathbf{e}_m^T$
equals to $|\mathbf{e}_m^T\mathbf{A}^{-1}\mathbf{a}|=|\mathrm{R}_m(\infty)|,$
it follows
from Assumption \ref{ASS} that
$\mathbf{A}+\zeta\mathbf{a}\mathbf{e}_m^T=\mathbf{A}(\mathbf{E}_m
+\zeta\mathbf{A}^{-1}\mathbf{a}\mathbf{e}_m^T)$ is invertible
for $|\zeta|<1.$
We can now safely define the discretized differential symbol
$\Delta(\zeta)$ for
BGA  as
\begin{align*}
\Delta(\zeta):=(\mathbf{A}+\zeta\mathbf{a}\mathbf{e}_m^T)^{-1}
(\mathbf{L}+\zeta \mathbf{l}\mathbf{e}_m^T).
\end{align*}
Then Eq. \eqref{YGeq} turns to
\begin{align}\label{Yrep}
\left(\Delta(\zeta)/h-\lambda \mathbf{E}_m\right)\mathbf{Y}(\zeta)=\mathbf{G}(\zeta),
\end{align}
where $\mathbf{E}_m$ denotes the $m\times m$ identity matrix.

The next step is to prove that all eigenvalues of the matrix  $\Delta(\zeta)$
with $|\zeta|<1$
lie in the complex half-plane $\mathbb{C}^{+}.$
This is equivalent to demonstrating that  $\Delta(\zeta)-z \mathbf{E}_m$
is invertible for all $z\in \mathbb{C}^-$ and $|\zeta|<1.$
In fact, we can establish a more rigorous result, as shown below.

\begin{lemma}
\label{lemmaNS}
Suppose that BGA satisfies Assumption \ref{ASS},
$\Delta(\zeta)$
is its discretized differential symbol,
and its stability domain, $\mathbb{D},$ is denoted by
\begin{align*}
\mathbb{D}:=\{z\in \mathbb{C}:|\mathrm{R}_m(z)|\leq 1\}.
\end{align*}
Then for any  $z\in \mathbb{D}$ and $|\zeta|<1,$
$\Delta(\zeta)-z \mathbf{E}_m$ is invertible
and its inverse can be represented by
\begin{align*}
(\Delta(\zeta)-z \mathbf{E}_m)^{-1}
=&(\mathbf{L}-z\mathbf{A})^{-1}\mathbf{A}
+\mathrm{R}_m(z)\mathbf{B}(\mathbf{L}-z\mathbf{A})^{-1}\mathbf{A}\zeta
\\
&+\sum_{j=2}^{\infty}((\mathrm{R}_m(z))^{j}\mathbf{B}
(\mathbf{L}-z\mathbf{A})^{-1}\mathbf{A}+
(\mathrm{R}_m(z))^{j-1}\mathbf{B}
(\mathbf{L}-z\mathbf{A})^{-1}\mathbf{a}\mathbf{e}_m^T)\zeta^j,
\end{align*}
where $\mathbf{B}:=(\mathbf{L}-z\mathbf{A})^{-1}
(z\mathbf{a}-\mathbf{l})\mathbf{e}_m^T.$
\end{lemma}
\begin{proof}
Direct computation implies
for any $z\in  \mathbb{D},$
\begin{align*}
\Delta(\zeta)-z \mathbf{E}_m
=&(\mathbf{A}+\zeta\mathbf{a}\mathbf{e}_m^T)^{-1}
((\mathbf{L}+\zeta \mathbf{l}\mathbf{e}_m^T)
-z(\mathbf{A}+\zeta\mathbf{a}\mathbf{e}_m^T))
\\
=&(\mathbf{A}+\zeta\mathbf{a}\mathbf{e}_m^T)^{-1}
((\mathbf{L}-z\mathbf{A})
+\zeta(\mathbf{l}\mathbf{e}_m^T-z\mathbf{a}\mathbf{e}_m^T)).
\end{align*}
Because all eigenvalues of $\mathbf{A}^{-1}\mathbf{L}$
reside in the domain $\mathbb{C}\setminus\mathbb{D},$
it follows that the matrix $\mathbf{L}-z\mathbf{A}$ is invertible.
Consequently, we obtain
\begin{align*}
\Delta(\zeta)-z \mathbf{E}_m
=
(\mathbf{A}+\zeta\mathbf{a}\mathbf{e}_m^T)^{-1}
(\mathbf{L}-z\mathbf{A})
(\mathbf{E}_m - \zeta \mathbf{B}).
\end{align*}
For $j\geq 2,$ direct computation yields
\begin{align*}
\mathbf{B}^j
=\mathrm{R}_m(z)\mathbf{B}^{j-1}
=(\mathrm{R}_m(z))^{j-1}\mathbf{B}.
\end{align*}
Since $|\mathrm{R}_m(z)|\leq1$ holds for any
$z\in \mathbb{D},$ we get the inequality
$|\mathrm{R}_m(z)\zeta|<1$ for all
$|\zeta|<1,$ which implies
\begin{align*}
\|\mathbf{B}\zeta\|^j
\leq |\mathrm{R}_m(z)\zeta|^{j-1}
\|\mathbf{B}\zeta\|
< 1
\end{align*}
for sufficiently large $j.$
Consequently, the series
$ \sum_{j=0}^{\infty}(\mathbf{B} \zeta)^j$
is a Neumann series and
converges to the inverse of $\mathbf{E}_m - \zeta \mathbf{B}$
in the case of $|\zeta|<1.$
In conclusion, we have established that all eigenvalues of the matrix
$\Delta(\zeta)$ lie in the domain $\mathbb{C}\setminus\mathbb{D}$
for $|\zeta|<1.$
Furthermore, it follows by a direct calculation that
\begin{align*}
(\Delta(\zeta)-z \mathbf{E}_m)^{-1}
=&(\mathbf{L}-z\mathbf{A})^{-1}\mathbf{A}
+\mathrm{R}_m(z)\mathbf{B}(\mathbf{L}-z\mathbf{A})^{-1}\mathbf{A}\zeta
\\
&+\sum_{j=2}^{\infty}((\mathrm{R}_m(z))^{j}\mathbf{B}
(\mathbf{L}-z\mathbf{A})^{-1}\mathbf{A}+
(\mathrm{R}_m(z))^{j-1}\mathbf{B}
(\mathbf{L}-z\mathbf{A})^{-1}\mathbf{a}\mathbf{e}_m^T)\zeta^j.
\end{align*}
This completes the proof.
\end{proof}

According to Cauchy's integral formula and Lemma \ref{lemmaNS},
we have for sufficiently small $h$
\begin{align*}
\mathbf{U}(\zeta)=\frac{1}{2\pi \mathrm{i}}\int_{\Gamma}
K(\lambda)\left(\Delta(\zeta)/h-\lambda \mathbf{E}_m\right)^{-1}\mathbf{G}(\zeta)
d\lambda=K\left(\Delta(\zeta)/h\right)\mathbf{G}(\zeta),
\end{align*}
where
$\mathbf{U}(\zeta):=\sum_{n=0}^{\infty}\mathbf{U}_n\zeta^n,$
$\mathbf{U}_n:=((K(\partial_t^h)g)(t^n_1)~~\cdots~~(K(\partial_t^h)g)(t^n_m))^T$
and $(K(\partial_t^h)g)(t^n_i)$ denotes the approximation to $(K(\partial_t)g)(t^n_i).$
Then BGACQ is defined by
\begin{align*}
\mathbf{U}_n=\sum_{j=0}^n\mathbf{W}_j\mathbf{G}_{n-j} ,~~n=0,1,\cdots,N,
\end{align*}
where $K\left(\Delta(\zeta)/h\right)=\sum_{n=0}^{\infty}\mathbf{W}_n\zeta^n.$
A direct calculation leads to
\begin{align*}
(K(\partial_t^h)g)(t^n_i)=
\mathbf{e}_i^T\sum_{j=0}^{n}\mathbf{W}_j\mathbf{G}_{n-j}.
\end{align*}
In this paper, we use the composite trapezoid rule to numerically compute the weight matrix $\mathbf{W}_j,$ that is,
\begin{align*}
\mathbf{W}_j =\frac{1}{2\pi \mathrm{i}}
\int_{|\zeta|=\rho}K (\Delta(\zeta)/h )
\zeta^{-j-1}d\zeta
\approx\frac{\rho^{-j}}{L}
\sum_{l=0}^{L-1}
K (\Delta (\rho \mathrm{e}^{\mathrm{i}l2\pi/L} )/h )
\mathrm{e}^{-\mathrm{i}jl2\pi/L},
\end{align*}
where $\rho=\mathrm{tol}^{1/6N},$
$L=5N$ and $\mathrm{tol} = 10^{-16}.$
To achieve efficient computation, the above summations are implemented using the fast Fourier transform.
In order to evaluate the matrix function $K (\Delta (\rho \mathrm{e}^{\mathrm{i}l2\pi/L} )/h ),$
we first
compute the invertible matrix $\mathbf{P}$ and
the diagonal matrix $ \mathrm{diag}(d_1,\cdots,d_m)$ to
diagonalize the matrix $\Delta (\rho \mathrm{e}^{\mathrm{i}l2\pi/L} ),$ that is,
\begin{align*}
\Delta (\rho \mathrm{e}^{\mathrm{i}l2\pi/L} )
=\mathbf{P} \mathrm{diag}(d_1,\cdots,d_m)\mathbf{P}^{-1}.
\end{align*}
Then  $K (\Delta (\rho \mathrm{e}^{\mathrm{i}l2\pi/L} )/h )$
is evaluated  by
\begin{align*}
K (\Delta (\rho \mathrm{e}^{\mathrm{i}l2\pi/L} )/h )
=\mathbf{P} \mathrm{diag}(K(d_1/h),\cdots,K(d_m/h))\mathbf{P}^{-1}.
\end{align*}
We have performed numerical computations of $\Delta (\rho \mathrm{e}^{\mathrm{i}l2\pi/L} )$
for several BGAs satisfying Assumption  \ref{ASS}.
The results show that all these approaches lead to diagonalizable matrices.
Furthermore, the condition numbers of the eigenvalues of
$\Delta\left(\rho \mathrm{e}^{\mathrm{i}l2\pi/L}\right)$
are  numerically found to be between one and two, indicating that the computations are not severely affected by instabilities.
However, evaluating $K (\Delta (\rho \mathrm{e}^{\mathrm{i}l2\pi/L} )/h )$
introduces rounding errors of order $\mathcal{O}(h^{-1}).$

Let us turn to investigate the convergence property of BGACQ as the step size
$h$ of the coarse grid $\mathcal{X}_N$ tends to $0,$ while the parameters
$k_1,k_2,m$  remain fixed.
Following  \cite{lubich1988convolution1}, we achieve this by integrating the product of the error bound for BGA and
 $K(\lambda)$ over the curve $\Gamma.$
For $n=0,\cdots,N-1,$ we have
\begin{align}
\label{IVPexact}
(\mathbf{L}-h \mathbf{A})\mathbf{Y}_{n}^{ref}
= (h\lambda \mathbf{a}-\mathbf{l})\mathbf{e}_m^T\mathbf{Y}_{n-1}^{ref} + h\mathbf{a}\mathbf{e}_m^T\mathbf{G}_{n-1}
+h\mathbf{A}\mathbf{G}_{n} +\mathbf{Q}_n,
\end{align}
where $\mathbf{Q}_n=(q^n_1~~\cdots~~q^n_m)^T.$
Combining Eqs. \eqref{IVPnum} and \eqref{IVPexact} yields
\begin{align*}
(\mathbf{L}-h\lambda \mathbf{A})(\mathbf{Y}_{n}^{ref} - \mathbf{Y}_{n} )
= (h\lambda \mathbf{a}-\mathbf{l})\mathbf{e}_m^T(\mathbf{Y}_{n-1}^{ref}- \mathbf{Y}_{n-1} )
+\mathbf{Q}_n.
\end{align*}
Denoting the collocation error $\epsilon^n_j := y(t^n_j)-y^n_j$
with $j=1,\cdots,m,$
we obtain
\begin{align*}
\epsilon^n_m=\mathrm{R}_m(h\lambda)\epsilon^{n-1}_m
+\mathbf{e}_m^T(\mathbf{L}-h\lambda \mathbf{A})^{-1}\mathbf{Q}_n.
\end{align*}
An iterative calculation results in
\begin{align*}
\epsilon^n_m=\sum_{j=0}^{n}(\mathrm{R}_m(h\lambda))^j\mathbf{e}_m^T
(\mathbf{L}-h\lambda \mathbf{A})^{-1}\mathbf{Q}_{n-j}.
\end{align*}
By leveraging the representation of the collocation errors of BGA for IVP \eqref{IVP}, we establish the convergence order of BGACQ for CI \eqref{CI} when  $\mu<0.$

\begin{theorem}\label{mainthm}
Suppose that
\begin{itemize}
  \item  The function $g$ in CI \eqref{CI} possesses a continuous derivative of at least order $k_1+k_2+2$
         over $[0,T]$ and additionally satisfies the initial condition
         $$g(0)=g^{(1)}(0)=\cdots=g^{(k_1+k_2+1)}(0)=0;$$
  \item The underlying ODE solver, BGA,  satisfies Assumption \ref{ASS};
\item $K(\lambda)$ is analytic in the half-plane $\mathbb{C}_{\sigma}$ and
bounded as $|K(\lambda)|\leq M|\lambda|^{\mu}$ with a positive constant $M$ and a negative real number $\mu.$
\end{itemize}
Then, for sufficiently large $N$ and any $n$ satisfying $0\leq n\leq N-1,$
it follows that
\begin{align*}
\|\mathbf{Err}^{\mathrm{CQ}}_{n,N}\|\leq
C h^{k_1+k_2+2}.
\end{align*}
where
\begin{align*}
\mathbf{Err}^{\mathrm{CQ}}_{n,N}:=
(
   (K(\partial_t)g)(t^n_1)-(K(\partial_t^h)g)(t^n_1)~,~\cdots~,~
 (K(\partial_t)g)(t^n_m)-(K(\partial_t^h)g)(t^n_m)
  )^T.
\end{align*}
\end{theorem}

\begin{proof}
Let us split the curve $\Gamma$ into three parts,
\begin{align*}
\Gamma_1:=\{\lambda\in \Gamma:|\lambda|\leq 1\},
\Gamma_2:=\{\lambda\in \Gamma:h\leq|h\lambda|\leq 1\},
\Gamma_3:=\{\lambda\in \Gamma:|h\lambda|\geq 1\}.
\end{align*}
For sufficiently small $h,$
Lemma \ref{AdaErr} guarantees
\begin{align*}
\|\mathbf{Q}_{n-j}\|\leq Ch^{k_1+k_2+3} ~~\mathrm{for~~any}~~ \lambda \in \Gamma_1.
\end{align*}
Moreover, since $0$ is not an eigenvalue of $\mathbf{A}^{-1}\mathbf{L},$
$\|\mathbf{e}_m^T
(\mathbf{L}-h\lambda \mathbf{A})^{-1}\|$ remains bounded by $C.$
Applying BGA  to $y'(t)=\lambda y(t)$ with the initial value $y(0)=1$
yields
\begin{align}\label{approxEXP}
\mathrm{R}_m(h\lambda) = \mathrm{e}^{h\lambda}+\mathcal{O}((h\lambda)^{k_1+k_2+3})
~~\mathrm{as}~~h\lambda\rightarrow 0.
\end{align}
As a result, we have $|\mathrm{R}_m(h\lambda)|^j$ is
bounded by $C$ when $\lambda\in \Gamma_1$
and $h$ is sufficiently small.
Therefore, it follows
\begin{align}
\label{ErrG1}
\left|\int_{\Gamma_1}K(\lambda)\epsilon^n_md\lambda\right|
\leq Ch^{k_1+k_2+2}.
\end{align}
Let's now consider the case where $\lambda \in \Gamma_2.$
Combining Eq. \eqref{approxEXP} with Assumption \ref{ASS} gives
\begin{align*}
\frac{|z|(1-|\mathrm{R}_m(z)|^{n+1})}{1-|\mathrm{R}_m(z)|}\leq C~~\mathrm{for}~~|z|\leq 1.
\end{align*}
Additionally, $\|\mathbf{e}_m^T
(\mathbf{L}-h\lambda \mathbf{A})^{-1}\|$ remains bounded by $C.$
With the help of Lemma \ref{AdaErr}, we have
\begin{align}\label{ErrG2}
\nonumber
\left|\int_{\Gamma_2}K(\lambda)\epsilon^n_md\lambda\right|
\leq& \int_{\Gamma_2}|K(\lambda)|
\frac{1-|\mathrm{R}_m(h\lambda)|^{n+1}}{1-|\mathrm{R}_m(h\lambda)|}
\|\mathbf{e}_m^T
(\mathbf{L}-h\lambda \mathbf{A})^{-1}\|\|\mathbf{Q}_{n-j}\|
|d\lambda|
\\ \nonumber
\leq &Ch^{k_1+k_2+2}
\int_{h\Gamma_2}|K(w/h)||w|^{-1}
\frac{|w|(1-|\mathrm{R}_m(w)|^{n+1})}{1-|\mathrm{R}_m(w)|}
|dw|
\\
\leq &Ch^{k_1+k_2+2-\mu}(1+h^{\mu})\leq Ch^{k_1+k_2+2}.
\end{align}
where the variable transformation $w=h\lambda$ is employed.
In the case of $\lambda \in \Gamma_3,$
it follows that $|\mathrm{R}_m(h\lambda)|\leq \rho<1,$
which implies $\sum_{j=0}^{n}|\mathrm{R}_m(h\lambda)|^j$ remains bounded
by $C.$
Meanwhile, we have the product $\|(\mathbf{L}-h\lambda \mathbf{A})^{-1}\||h\lambda|$ is also
bounded by $C.$
Following Lemma \ref{AdaErr}, this implies that
\begin{align}
\label{ErrG3}
\nonumber
\left|\int_{\Gamma_3}K(\lambda)\epsilon^n_md\lambda\right|\leq&
C\int_{\Gamma_3}\frac{|\lambda|^{\mu}}{|h\lambda|}
\left(\sum_{j=0}^{n}|\mathrm{R}_m(h\lambda)|^j\right)
\left(\|(\mathbf{L}-h\lambda \mathbf{A})^{-1}\||h\lambda|\right)
\|\mathbf{Q}_{n-j}\|
|d\lambda|
\\
\leq& Ch^{k_1+k_2+2}\int_{\Gamma_3}|\lambda|^{\mu-1}|d\lambda|\leq
Ch^{k_1+k_2+2-\mu}
\end{align}

Combining Eqs. \eqref{ErrG1}, \eqref{ErrG2} and \eqref{ErrG3} gives
\begin{align*}
|(K(\partial_t)g)(t^n_m)-(K(\partial_t^h)g)(t^n_m)|
\leq C h^{k_1+k_2+2}.
\end{align*}
Letting $\mathrm{R}_{j}(z)=\mathbf{e}_j^{T}(\mathbf{L}-z \mathbf{A})^{-1}
(z \mathbf{a}-\mathbf{l}),$
 it follows by direct computation
\begin{align*}
\epsilon_j^n=&\frac{\mathrm{R}_{j}(h\lambda)}
{\mathrm{R}_{m}(h\lambda)}
\mathrm{R}_{m}(h\lambda)\epsilon^{n-1}_m
+\mathbf{e}_j^T(\mathbf{L}-h\lambda \mathbf{A})^{-1}\mathbf{Q}_n
\\
=&\frac{\mathrm{R}_{j}(h\lambda)}{\mathrm{R}_{m}(h\lambda)}
\sum_{j=1}^n(\mathrm{R}_{m}(h\lambda))^j
\mathbf{e}_m^T(\mathbf{L}-h\lambda \mathbf{A})^{-1}\mathbf{Q}_{n-j}
+\mathbf{e}_j^T(\mathbf{L}-h\lambda \mathbf{A})^{-1}\mathbf{Q}_n.
\end{align*}
Since the ratio $\mathrm{R}_{j}(h\lambda)/\mathrm{R}_{m}(h\lambda)$ remains bounded by $C$ for any $\lambda\in \Gamma,$
we   deduce that the convergence orders of the numerical errors on the fine grids coincide with those on the coarse grids.
This implies that the convergence of BGACQ is uniform across the entire grid,
that is,
\begin{align*}
\|\mathbf{Err}^{\mathrm{CQ}}_{n,N}\|\leq
C h^{k_1+k_2+2}.
\end{align*}
This completes the proof.
\end{proof}

In the case of $\mu\geq 0,$
we define $r$ to be the smallest integer exceeding $\mu,$
and let $K_r(\lambda)=K(\lambda)/\lambda^r.$
It follows by a direct calculation that
\begin{align*}
(K(\partial_t)g)(t^{n}_i)-(K(\partial_t^h)g)(t^{n}_i)
=&(K_r(\partial_t)(\partial_t)^rg)(t^{n}_i)
-(K_r(\partial^h_t)(\partial_t)^rg)(t^{n}_i)
\\
&+(K(\partial^h_t)(\partial_t^h)^{-r}(\partial_t)^rg)(t^{n}_i)
-(K(\partial_t^h)g)(t^{n}_i).
\end{align*}
Suppose that $g$ satisfies
\begin{align*}
g(0)=g^{(1)}(0)=\cdots=g^{(r+k_1+k_2+1)}(0)=0.
\end{align*}
By application of Theorem \ref{mainthm}, we have that
\begin{align}\label{Diff1}
|(K_r(\partial_t)(\partial_t)^rg)(t^{n}_i)
-(K_r(\partial^h_t)(\partial_t)^rg)(t^{n}_i)|\leq Ch^{k_1+k_2+2}.
\end{align}
Note that
\begin{align*}
K(\Delta(\zeta)/h)=\frac{1}{2\pi \mathrm{i}}
\int_{\Gamma'}K(\lambda/h)(\Delta(\zeta)-\lambda \mathbf{E}_m)^{-1}d\lambda,
\end{align*}
where
$\Gamma'$ comprise a circular arc $|\lambda|=r$ and $\mathrm{Re}(\lambda)\geq h\sigma$ such that it falls in the stability domain of the ODE
solver.
Based on Lemma \ref{lemmaNS}, we obtain the following derivation
\begin{align*}
\sum_{n=0}^N\|\mathbf{W}_n\|\leq &
C\int_{\Gamma'}|K(\lambda/h)|\sum_{n=0}^{N}
|\mathrm{R}_m(\lambda)|^{n}|d\lambda|.
\end{align*}
Split $\Gamma'$ into two parts,
\begin{align*}
\Gamma'_1:=\{\lambda\in\Gamma' : |\lambda|\leq h \}
~~ \mathrm{and} ~~ \Gamma'_2:=\Gamma'\setminus \Gamma'_1.
\end{align*}
Then we can obtain the result using an analogous argument to the proof of Theorem \ref{mainthm}
\begin{align*}
\int_{\Gamma'_1}|K(\lambda/h)|\sum_{n=0}^{N}
|\mathrm{R}_m(\lambda)|^{n}|d\lambda|
\leq C,
\end{align*}
and
\begin{align*}
\int_{\Gamma'_2}|K(\lambda/h)|\sum_{n=0}^{N}
|\mathrm{R}_m(\lambda)|^{n}\||d\lambda|
=&\int_{\Gamma'_2}|K(\lambda/h)||\lambda|^{-1}\left(|\lambda|\sum_{n=0}^{N}
|\mathrm{R}_m(\lambda)|^{n}\right)|d\lambda|
\\
\leq &Ch^{-\mu}\int_{\Gamma'_2}|\lambda|^{\mu-1}|d\lambda|
\leq C\left\{
        \begin{array}{ll}
          |\log(h)|, & \mu=0 ,\\
          h^{-\mu}, &  \mu>0 .
        \end{array}
      \right.
\end{align*}
Therefore, we get
\begin{align}
\label{boundforweight}
\sum_{n=0}^N\|\mathbf{W}_n\|\leq C\left\{
        \begin{array}{ll}
          |\log(h)|, & \mu=0 ,\\
          h^{-\mu}, &  \mu>0 .
        \end{array}
      \right.
\end{align}
The difference $((\partial_t^h)^{-r}(\partial_t)^rg)(t)-g(t)$
is just the local error of BGA applied to ODE $y^{(r)}(t)=g^{(r)}(t)$
with zero initial values (see \cite{banjai2011an}).
Thus we have
\begin{align*}
|((\partial_t^h)^{-r}(\partial_t)^rg-g)(t^n_i)|\leq Ch^{k_1+k_2+3}.
\end{align*}
Letting $\delta \mathbf{G}_n:=(
                                ((\partial_t^h)^{-r}(\partial_t)^rg-g)(t^n_1)~,~    \cdots~,~    ((\partial_t^h)^{-r}(\partial_t)^rg-g)(t^n_m)
                             )^T,$
 we get
\begin{align}\label{Diff2}\nonumber
&|(K(\partial^h_t)(\partial_t^h)^{-r}(\partial_t)^rg)(t^{n}_i)
-(K(\partial_t^h)g)(t^{n}_i)|
\\
\leq&\|\mathbf{e}_i^T\|\sum_{j=0}^n\|\mathbf{W}_{n-j}\|
\|\delta \mathbf{G}_j\|
\leq
C\left\{
        \begin{array}{ll}
          h^{k_1+k_2+3}|\log(h)|, & \mu=0 ,\\
          h^{k_1+k_2+3-\mu}, &  \mu>0 .
        \end{array}
      \right.
\end{align}
Now we get the convergence property of BGACQ
by combining Eqs. \eqref{Diff1} and \eqref{Diff2} in the case of $\mu\geq 0.$

\begin{theorem}\label{mainthm2}
Suppose that
\begin{itemize}
  \item  The function $g$ in CI \eqref{CI} possesses a continuous derivative of at least order $r+k_1+k_2+2$
         over $[0,T]$ and additionally satisfies the initial condition
         $$g(0)=g^{(1)}(0)=\cdots=g^{(r+k_1+k_2+1)}(0)=0,$$
         where $r$ denotes the smallest integer exceeding $\mu;$
  \item The underlying ODE solver, BGA,  satisfies Assumption
\ref{ASS};
\item $K(\lambda)$ is analytic in the half-plane $\mathbb{C}_{\sigma}$ and
bounded as $|K(\lambda)|\leq M|\lambda|^{\mu}$ with a positive constant $M$ and a nonnegative real number $\mu.$
\end{itemize}
Then, for sufficiently large $N$ and any $n$ satisfying $0\leq n\leq N-1,$
it follows that
\begin{align*}
\|\mathbf{Err}^{\mathrm{CQ}}_{n,N}\|\leq
Ch^{\min\{k_1+k_2+3-\mu , k_1+k_2+2\}}.
\end{align*}
\end{theorem}

In solving time-domain boundary integral equations
for some acoustic scattering problems, the dissipative property of CQ becomes critical.
Lower dissipation in these methods generally translates to improved computational efficiency.
As discussed for RKCQ in \cite[Chapter 5]{banjai2022integral},
discretizing the time-domain boundary integral equation using BGACQ involves replacing the free-space Green's function kernel
$\mathcal{G}(\mathbf{x},\lambda)$ with $\mathcal{G}(\mathbf{x},\hat{\lambda}).$
Here, for sufficiently small $\hat{\lambda}h,$
$\hat{\lambda}$ is the solution of
$\mathrm{R}_m(\hat{\lambda}h)=\mathrm{e}^{\lambda h}.$
Analyzing the numerical scheme's dissipation then involves expanding
$\mathrm{R}_m^{-1}(\mathrm{e}^{-\mathrm{i}\omega})$ to obtain the approximation $\hat{\lambda}$
to the frequency
$\lambda = \mathrm{i}\omega.$
Since
\begin{align*}
\mathrm{R}_m'(z)=\mathbf{e}_m^T((\mathbf{L}-z\mathbf{A})^{-1}
\mathbf{A}(\mathbf{L}-z\mathbf{A})^{-1})(z\mathbf{a}-\mathbf{l})+
\mathbf{e}_m^T(\mathbf{L}-z\mathbf{A})^{-1}\mathbf{a},
\end{align*}
we have $\mathrm{R}_m'(0)=\mathbf{e}_m^T
\mathbf{L}^{-1}(-\mathbf{A}\mathbf{L}^{-1}\mathbf{l}+\mathbf{a})=1,$
which implies
the inverse $\mathrm{R}_m^{-1}$ around origin is well-defined.
We have  computed the difference between
$\lambda$ and $\hat{\lambda}$  numerically for various BGACQs,
and present the results in the following,
\begin{align*}
\mathrm{BGACQ}_{0,1}^3:~&\hat{\lambda}=\lambda
+5.1\times 10^{-4}\omega(\omega h)^{3}
+7.6\times 10^{-4}\mathrm{i}\omega(\omega h)^{4}+\omega \mathcal{O}((\omega h)^{5}),
\\
\mathrm{BGACQ}_{0,2}^4:~&\hat{\lambda}=\lambda
-6.2\times 10^{-5}\mathrm{i}\omega(\omega h)^{4}
+6.7\times 10^{-5}\omega(\omega h)^{5}+\omega \mathcal{O}((\omega h)^{6}),
\\
\mathrm{BGACQ}_{1,2}^5:~&\hat{\lambda}=\lambda
+4.9\times 10^{-7}\omega(\omega h)^{5}
-8.8\times 10^{-8}\mathrm{i}\omega(\omega h)^{6}+\omega \mathcal{O}((\omega h)^{7}),
\\
\mathrm{BGACQ}_{0,1}^{12}:~&\hat{\lambda}=\lambda
+2.0\times 10^{-5}\omega(\omega h)^{3}
+2.1\times 10^{-5}\mathrm{i}\omega(\omega h)^{4}+\omega \mathcal{O}((\omega h)^{5}),
\\
\mathrm{BGACQ}_{0,2}^{12}:~&\hat{\lambda}=\lambda
-1.1\times 10^{-6}\mathrm{i}\omega(\omega h)^{4}
+1.2\times 10^{-6}\omega(\omega h)^{5}+\omega \mathcal{O}((\omega h)^{6}),
\\
\mathrm{BGACQ}_{1,2}^{12}:~&\hat{\lambda}=\lambda
-2.0\times 10^{-8}\omega(\omega h)^{5}
-2.0\times 10^{-8}\mathrm{i}\omega(\omega h)^{6}+\omega \mathcal{O}((\omega h)^{7}).
\end{align*}
We can see the dissipation of BGACQ is demonstrably small and diminishes further as the block size  increases.

If the initial conditions for $g$ in Theorems \ref{mainthm}
and \ref{mainthm2} are not met, the convergence order of
BGACQ will be reduced.
To overcome this limitation, modifications to BGACQ should be explored.
We first focus on the case with $\mu<0.$
For a general $g,$ applying Taylor's expansion yields
\begin{align*}
g(t)=&\sum_{l=0}^{k_1+k_2+1}\frac{t^l}{l!}g^{(l)}(0)+
\frac{1}{(k_1+k_2+1)!}\int_0^tg^{(k_1+k_2+2)}(s)(t-s)^{k_1+k_2+1}ds
\\
=&\sum_{l=0}^{k_1+k_2+1}\frac{g^{(l)}(0)}{l!}p_l(t)+\tilde{p}(t),
\end{align*}
where
\begin{align*}
p_l(t):=t^l,~~
\tilde{p}(t):= \frac{1}{(k_1+k_2+1)!}\int_0^tg^{(k_1+k_2+2)}(s)(t-s)^{k_1+k_2+1}ds.
\end{align*}
We solve the following equations for $i=1,\cdots,m$
and $l=0,\cdots,k_1+k_2+1,$
\begin{align}
\label{Modeq}
(K(\partial_t)p_l)(t^{n}_i)=
\mathbf{e}_i^T\sum_{j=0}^{n}\mathbf{W}_{n-j}\mathbf{p}_l^j
+\sum_{j=0}^{k_1+k_2+1}\omega_{i,j}^np_l(t^0_j),
\end{align}
where $\mathbf{p}_l^j:=(p_l(t^j_1)~,~\cdots~,~p_l(t^j_m))^T.$
Letting
\begin{align*}
\mathbf{w}_i:=&\left(
                \begin{array}{c}
                  w^n_{i,0} \\
                   w^n_{i,1} \\
                  \vdots \\
                   w^n_{i,k_1+k_2+1} \\
                \end{array}
              \right),
\mathbf{K}_i:=\left(
                \begin{array}{c}
                  (K(\partial_t)p_0)(t^{n}_i) \\
                  (K(\partial_t)p_1)(t^{n}_i) \\
                  \vdots \\
                   (K(\partial_t)p_{k_1+k_2+1})(t^{n}_i) \\
                \end{array}
              \right),
\\
\mathbf{V}:=&\left(
      \begin{array}{cccc}
        p_0(t^0_0) & p_0(t^0_1) & \cdots & p_0(t^0_{k_1+k_2+1}) \\
        p_1(t^0_0) & p_1(t^0_1) & \cdots & p_1(t^0_{k_1+k_2+1})\\
        \vdots & \vdots &  & \vdots \\
        p_{k_1+k_2+1}(t^0_0) & p_{k_1+k_2+1}(t^0_1) & \cdots & p_{k_1+k_2+1}(t^0_{k_1+k_2+1}) \\
      \end{array}
    \right),
\end{align*}
Eq. \eqref{Modeq} reduces to
\begin{align*}
\mathbf{V}\mathbf{w}_i=\mathbf{K}_i-
\left(
  \begin{array}{c}
    \mathbf{e}_i^T\sum_{j=0}^{n}\mathbf{W}_{n-j}\mathbf{p}_0^j  \\
    \vdots  \\
     \mathbf{e}_i^T\sum_{j=0}^{n}\mathbf{W}_{n-j}\mathbf{p}_{k_1+k_2+1}^j \\
  \end{array}
\right).
\end{align*}
Following a similar derivation as for Eq. \eqref{boundforweight},
we obtain
\begin{align*}
\sum_{j=0}^{n}\|\mathbf{W}_{n-j}\|\leq C ~~\mathrm{for}~~\mu<0~~\mathrm{as}~~
h\rightarrow 0.
\end{align*}
Since
$ (K(\partial_t)p_l)(t^{n}_i)=\mathcal{O}(h^{l})$
and $ p_l(t^0_j)=\mathcal{O}(h^{l}),$
it follows that  $\|\mathbf{w}_i\|=\mathcal{O}(1).$
Then we define the  modified BGACQ (MBGACQ) by
\begin{align*}
(K(\bar{\partial}_t^h)g)(t^n_{i}):=
(K(\partial_t^h)g)(t^n_{i})
+\mathbf{w}_i^T\hat{\mathbf{G}}_0,
\end{align*}
where $\hat{\mathbf{G}}_0:=(g(t^0_0)~,~\cdots~,~g(t^0_{k_1+k_2+1}))^T.$
We can now express the quadrature error as
\begin{align*}
(K(\bar{\partial}_t^h)g)(t^n_{i})-(K(\partial_t)g)(t^n_{i})
=&(K(\bar{\partial}_t^h)\tilde{p})(t^n_{i})-(K(\partial_t)\tilde{p})(t^n_{i})
\\
=&(K(\partial_t^h)\tilde{p})(t^n_{i})-(K(\partial_t)\tilde{p})(t^n_{i})
+\mathbf{w}_i^T\hat{\mathbf{p}}_0,
\end{align*}
where $\hat{\mathbf{p}}_0:=(\tilde{p}(t^0_0)~,~\cdots~,~\tilde{p}(t^0_{k_1+k_2+1}))^T.$
Because $\|\hat{\mathbf{p}}_0\|=\mathcal{O}(h^{k_1+k_2+2})$ as $h\rightarrow 0$
and $\tilde{p}(t)$ satisfies the initial condition in Theorem \ref{mainthm},
both $(K(\partial_t^h)\tilde{p})(t^n_{i})-(K(\partial_t)\tilde{p})(t^n_{i})$
and $\mathbf{w}_i^T\hat{\mathbf{p}}_0$ are $\mathcal{O}(h^{k_1+k_2+2}).$
This implies that MBGACQ recovers the full convergence order.
For the case of $\mu\geq 0,$
the initial condition $g(0)=\cdots=g^{(r)}(0)$
ensures the well-posedness of CI,
and the convergence order of MBGACQ can
be derived similarly to the case of  $\mu<0.$

\section{Numerical Experiments}

In this section, we present several numerical examples to validate the theoretical properties of BGACQ for calculation of CI \eqref{CI}.
For this purpose, we compute
the maximum of  quadrature errors,
$\|\mathbf{Err}^{\mathrm{CQ}}_{n,N}\|,$
in the $n-$th block
of BGACQ.
Furthermore,
we numerically compute the convergence order, denoted by ``$\mathrm{Order}$'', for the errors in each block.
The formula used is
\begin{align*}
\mathrm{Order}:=
\left|\frac{\log\|\mathbf{Err}^{\mathrm{CQ}}_{n,N_1}\|-
\log\|\mathbf{Err}^{\mathrm{CQ}}_{n,N_2}\|}{\log N_1-\log N_2}\right|.
\end{align*}
The first two examples demonstrate the convergence rates established in Theorems \ref{mainthm} and \ref{mainthm2}.
The third example assess the stability of BGACQ by numerically solving a class of integral equations with rapidly varying and conservative  solutions.
Finally, we consider the calculation of a family of oscillatory integrals.

\begin{example}\label{Ex1}
We compute  CI \eqref{CI} using BGACQ$_{0,1}^3,$
where the Laplace transform of its kernel is
\begin{align*}
K_{\mu}(\lambda)=\frac{\lambda^{\mu}}{1-\mathrm{e}^{-\lambda}}.
\end{align*}
The expression for this family of CIs can be derived as follows,
\begin{align*}
(K_{\mu}(\partial_t)g)(t)=\sum_{j=0}^{\infty}(\partial_t^{\mu} g)(t-j)~~
\mathrm{with}~~(\partial_t^{\mu} g)(t)=\frac{1}{\Gamma(\mu)}\int_0^t(t-s)^{-\mu-1}
g(s)ds.
\end{align*}
We select $g(t)=\mathrm{e}^{-0.4t}\sin^6 t$
to satisfy the initial conditions in Theorems \ref{mainthm} and \ref{mainthm2} for the case where $\mu\leq 2.$
Due to the difficulty of directly computing the exact value, we employ the value obtained using BGACQ$_{0,1}^3$ with $2^{10}$ blocks as a reference value.
\end{example}

Table \ref{Ex1T1} presents the computed results for the case of $\mu<0.$
As expected from Theorem \ref{mainthm}, the observed convergence order is approximately $3,$ which aligns well with the theoretical findings.
Table \ref{Ex1T2} summarizes the numerical results for $\mu\geq 0.$
We observe a computed convergence order of  $3$ for BGACQ$_{0,1}^3$
when $\mu=0$ or $0.8.$
This aligns with the theoretical prediction in Theorem \ref{mainthm2}.
For the case where $\mu=1.8,$
the computed order is $4-\mu=2.2,$
again consistent with the theoretical estimation in Theorem \ref{mainthm2}.

\begin{table}
\tabcolsep 0pt \caption{Maximum norms, $\|\mathbf{Err}_{N-1,N}^{\mathrm{CQ}}\|,$ of absolute errors of BGACQ$_{0,1}^{3}$
applied to Example \ref{Ex1} with $\mu<0.$}\label{Ex1T1}
\vspace*{-10pt}
\begin{center}
\def\temptablewidth{1\textwidth}
{\rule{\temptablewidth}{1pt}}
\begin{tabular*}{\temptablewidth}{@{\extracolsep{\fill}}ccccccc} \hline
               \multirow{2}*{$N$}& \multicolumn{2}{c}{$\mu = -0.2$} &
\multicolumn{2}{c}{$\mu = -0.8$}
& \multicolumn{2}{c}{$\mu = -1.8$}
\\
             & $\|\mathbf{Err}_{N-1,N}^{\mathrm{CQ}}\|$  & $\mathrm{Order}$
              & $\|\mathbf{Err}_{N-1,N}^{\mathrm{CQ}}\|$  & $\mathrm{Order}$
             & $\|\mathbf{Err}_{N-1,N}^{\mathrm{CQ}}\|$  & $\mathrm{Order}$
            \\ \hline
    $2^{3} $     &$2.8\times 10^{-4}$ &	--
    &$1.1\times 10^{-4}$ &	--
    &$9.8\times 10^{-6}$ &	-- \\
    $2^{4} $     &$3.1\times 10^{-5}$ &	3.2
    &$1.3\times 10^{-5}$ &	3.1
    &$5.3\times 10^{-7}$ &	4.2\\
    $2^{5}$     &$3.8\times 10^{-6}$ &	3.1
    &$1.5\times 10^{-6}$ &	3.1
    &$5.0\times 10^{-8}$ &	3.4\\
    $2^{6} $    &$4.6\times 10^{-7}$ &	3.0
    &$1.8\times 10^{-7}$ &	3.0
    &$5.3\times 10^{-9}$ &	3.2\\
    $2^{7}$     &$5.7\times 10^{-8}$ &	3.0
    &$2.2\times 10^{-8}$ &	3.0
    &$6.2\times 10^{-10}$ &	3.1 \\
    $2^{8}$     &$7.0\times 10^{-9}$ &	3.0
    &$2.7\times 10^{-9}$ &	3.0
    &$7.9\times 10^{-11}$ &	3.0 \\
     \hline
       \end{tabular*}
       {\rule{\temptablewidth}{1pt}}
       \end{center}
       \end{table}

\begin{table}
\tabcolsep 0pt \caption{Maximum norms, $\|\mathbf{Err}_{N-1,N}^{\mathrm{CQ}}\|,$ of absolute errors of BGACQ$_{0,1}^{3}$
applied to Example \ref{Ex1} with $\mu\geq0.$}\label{Ex1T2}
\vspace*{-10pt}
\begin{center}
\def\temptablewidth{1\textwidth}
{\rule{\temptablewidth}{1pt}}
\begin{tabular*}{\temptablewidth}{@{\extracolsep{\fill}}ccccccc} \hline
               \multirow{2}*{$N$}& \multicolumn{2}{c}{$\mu = 0$} &
\multicolumn{2}{c}{$\mu = 0.8$}
& \multicolumn{2}{c}{$\mu = 1.8$}
\\
             & $\|\mathbf{Err}_{N-1,N}^{\mathrm{CQ}}\|$  & $\mathrm{Order}$
              & $\|\mathbf{Err}_{N-1,N}^{\mathrm{CQ}}\|$  & $\mathrm{Order}$
             & $\|\mathbf{Err}_{N-1,N}^{\mathrm{CQ}}\|$  & $\mathrm{Order}$
            \\ \hline
    $2^{3}  $ &$3.6\times 10^{-4}$ &--
    &	$1.2\times 10^{-3}$ &--
    &$4.3\times 10^{-2}$ &	--  \\
    $2^{4} $  &$3.7\times 10^{-5}$ &3.3
    &$1.2\times 10^{-4}$ &	3.3
    &$5.6\times 10^{-3}$ &	2.9 \\
    $2^{5} $     &$4.3\times 10^{-6}$ &3.1
    &$1.6\times 10^{-5}$ &	2.8
    &$9.6\times 10^{-4}$ &	2.6 \\
    $2^{6} $     &$5.2\times 10^{-7}$ &	3.0
    &$2.1\times 10^{-6}$ &2.9
    &$1.9\times 10^{-4}$ &	2.4 \\
    $2^{7}$     &$6.4\times 10^{-8}$ &3.0
    &$2.6\times 10^{-7}$ &	3.0
    &$4.3\times 10^{-5}$ &	2.1 \\
    $2^{8} $    &$7.9\times 10^{-9}$ &3.0
    &$3.2\times 10^{-8}$ &	3.1
    &$9.1\times 10^{-6}$ &	2.2 \\
     \hline
       \end{tabular*}
       {\rule{\temptablewidth}{1pt}}
       \end{center}
       \end{table}

CQs have proven valuable tools for computing fractional derivatives and integrals.
The following example explores the application of the proposed BGACQ to the calculation of fractional integrals, demonstrating its effectiveness in this specific area.

\begin{example}
Consider calculation of CI in the form of
\begin{align}
\label{CI2ex}
\frac{1}{\Gamma(\alpha)}\int_0^t(t-s)^{\alpha-1}(\sin (s)+1)\mathrm{e}^{0.8s}ds~~\mathrm{with}~~t\in [0,5]~~\mathrm{and}~~\alpha\in (0,1).
\end{align}
\end{example}

Since the function $(\sin (s)+1)\mathrm{e}^{0.8s}$ does not vanish at $s=0,$
we utilize MBGACQ$_{0,1}^3,$ MBGACQ$_{0,2}^4,$
and MBGACQ$_{1,2}^5$ to compute CI \eqref{CI2ex}
and  verify the predicted convergence order.
As shown in Tables \ref{Ex2T1} ($\alpha=0.5$)
and \ref{Ex2T3} ($\alpha=0.9$),
the computed convergence orders are approximately
$3,4,5,$ respectively,
which aligns with the theoretical expectations established in Theorem \ref{mainthm}.
To elucidate the key difference between BGACQ$_{1,2}^5$
and MBGACQ$_{1,2}^5,$ Figure \ref{Ex2F1}
compares their absolute errors on the total grid for varying values of $N.$
We observe that the correction term incorporated in  MBGACQ$_{1,2}^5$
significantly reduces the numerical errors around $t=0.$

\begin{table}
\tabcolsep 0pt \caption{Maximum norms, $\|\mathbf{Err}_{N-1,N}^{\mathrm{CQ}}\|,$ of absolute errors computed by  MBGACQ$_{k_1,k_2}^{m^*}$ for CI \eqref{CI2ex} with $\alpha=0.5$.}\label{Ex2T1}
\vspace*{-10pt}
\begin{center}
\def\temptablewidth{1\textwidth}
{\rule{\temptablewidth}{1pt}}
\begin{tabular*}{\temptablewidth}{@{\extracolsep{\fill}}ccccccc} \hline
               \multirow{2}*{$N$}& \multicolumn{2}{c}{$(k_1,k_2)=(0,1)$} &
\multicolumn{2}{c}{$(k_1,k_2)=(0,2)$}
& \multicolumn{2}{c}{$(k_1,k_2)=(1,2)$}
\\
             & $\|\mathbf{Err}_{N-1,N}^{\mathrm{CQ}}\|$  & $\mathrm{Order}$
              & $\|\mathbf{Err}_{N-1,N}^{\mathrm{CQ}}\|$  & $\mathrm{Order}$
             & $\|\mathbf{Err}_{N-1,N}^{\mathrm{CQ}}\|$  & $\mathrm{Order}$
            \\ \hline
   8      &$1.2\times 10^{-2}$ &	--
           &$7.3\times 10^{-4}$&	--
         &$1.3\times 10^{-6}$ &	-- \\
    24     & $3.4\times 10^{-4}$ &	3.2
           &$8.9\times 10^{-6}$ &	4.0
        & $1.1\times 10^{-8}$ &	4.4\\
    40      &$6.7\times 10^{-5}$ &	3.2
           &$1.1\times 10^{-6}$ &	4.0
         &$8.0\times 10^{-10}$ &	5.1 \\
    56      &$2.3\times 10^{-5}$ &	3.1
           &$3.0\times 10^{-7}$ &	4.0
         &$1.4\times 10^{-10}$ &	5.1 \\
    72      &$1.1\times 10^{-5}$ &	3.1
           &$1.1\times 10^{-7}$ &	4.0
         &$3.8\times 10^{-11}$ &	5.2 \\
     \hline
       \end{tabular*}
       {\rule{\temptablewidth}{1pt}}
       \end{center}
       \end{table}

\begin{table}
\tabcolsep 0pt \caption{Maximum norms, $\|\mathbf{Err}_{N-1,N}^{\mathrm{CQ}}\|,$ of absolute errors computed by  MBGACQ$_{k_1,k_2}^{m^*}$ for CI \eqref{CI2ex} with $\alpha=0.9$.}\label{Ex2T3}
\vspace*{-10pt}
\begin{center}
\def\temptablewidth{1\textwidth}
{\rule{\temptablewidth}{1pt}}
\begin{tabular*}{\temptablewidth}{@{\extracolsep{\fill}}ccccccc} \hline
               \multirow{2}*{$N$}& \multicolumn{2}{c}{$(k_1,k_2)=(0,1)$} &
\multicolumn{2}{c}{$(k_1,k_2)=(0,2)$}
& \multicolumn{2}{c}{$(k_1,k_2)=(1,2)$}
\\
             & $\|\mathbf{Err}_{N-1,N}^{\mathrm{CQ}}\|$  & $\mathrm{Order}$
              & $\|\mathbf{Err}_{N-1,N}^{\mathrm{CQ}}\|$  & $\mathrm{Order}$
             & $\|\mathbf{Err}_{N-1,N}^{\mathrm{CQ}}\|$  & $\mathrm{Order}$
            \\ \hline
    8   & $1.5\times 10^{-2}$ &	--
    &$1.3\times 10^{-3}$ &	--
    & $1.7\times 10^{-6}$ &	--  \\
    24  &$4.3\times 10^{-4}$ &3.2
     &$1.5\times 10^{-5}$ &	4.0
     &$1.7\times 10^{-8}$ &	4.2 \\
    40   &$8.7\times 10^{-5}$ &3.1
    &$2.0\times 10^{-6}$ &	4.0
    & $1.5\times 10^{-9}$ &	4.8\\
    56  & $3.0\times 10^{-5}$ & 3.1
    &$5.1\times 10^{-7}$ &	4.0
    & $2.9\times 10^{-10}$ & 4.9 \\
    72  &$1.4\times 10^{-5}$ &	3.1
    &$1.9\times 10^{-7}$ &	4.0
     & $8.0\times 10^{-11}$ &	5.1 \\
     \hline
       \end{tabular*}
       {\rule{\temptablewidth}{1pt}}
       \end{center}
       \end{table}

\begin{figure}
\begin{center}
\includegraphics[width=14cm,height=6cm]{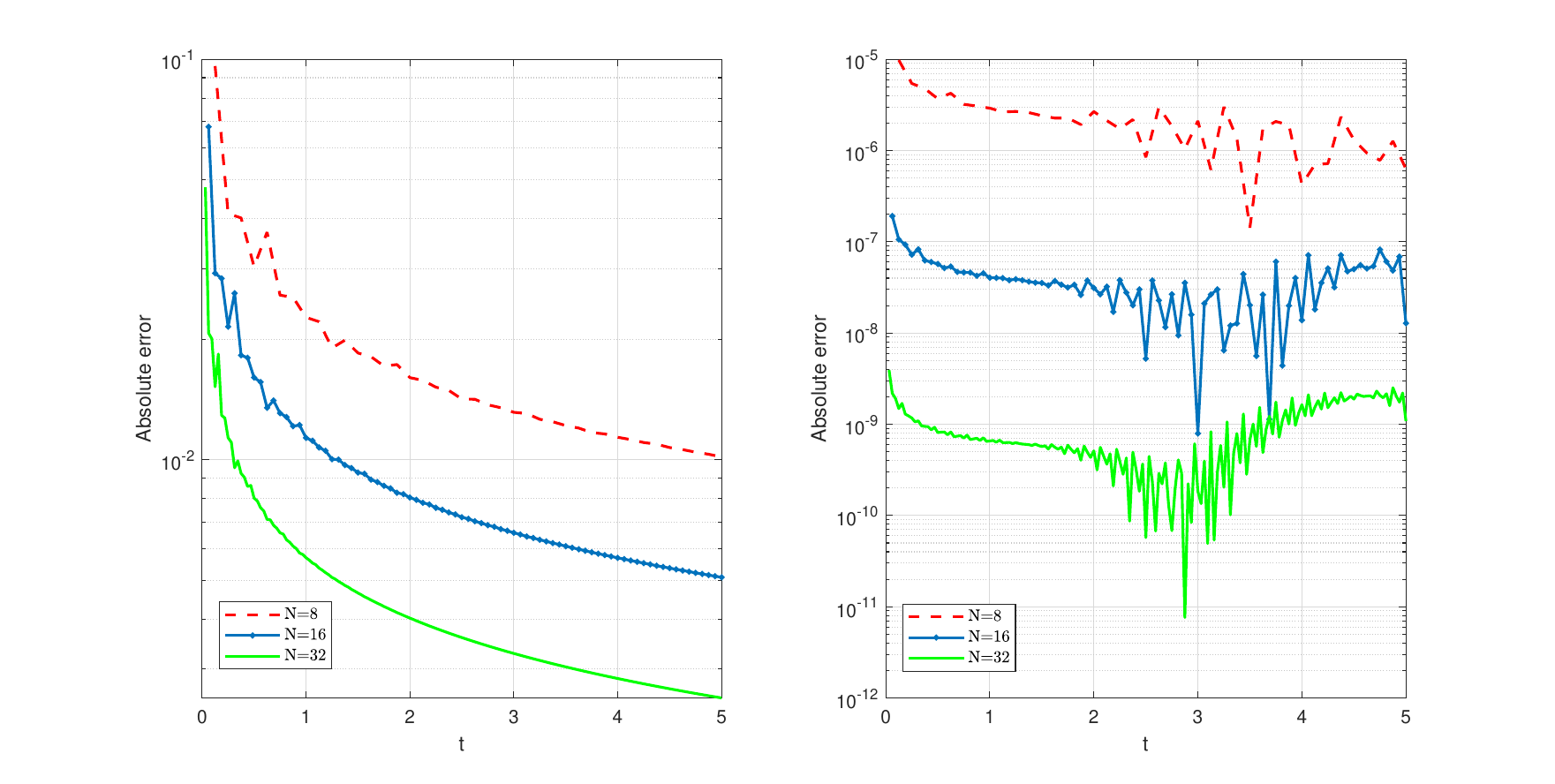}
\caption{Comparison between pointwise errors of BGACQ$_{1,2}^5$ (left)
and MBGACQ$_{1,2}^5$ (right) for CI \eqref{CI2ex} with
$\alpha=0.5$ and various $N.$}\label{Ex2F1}
\end{center}
\end{figure}

Now let us test the stability of the proposed CQ
by solving an integral equation with a fast varying
and conservative solution.

\begin{example}
Solve the integral equation
\begin{align}
\label{IEex}
\int_0^tk(t-s)u(s)ds=g(t)~~\mathrm{with}~~t\in [0,4].
\end{align}
Here the Laplace transform of $k(t)$ is selected to be
$K(\lambda)=1-\mathrm{e}^{-\lambda}$
and $g(t)=\mathrm{e}^{-100(t-0.5)^2}.$
Its exact solution is $u(t)=\sum_{j=0}^3g(t-j)$
(see \cite{banjai2022integral}).
Since direct calculations indicate that both
$u(t)$
and its derivatives  are negligible at $t=0,$
the correction term is omitted in the implementation of CQs for this case.
\end{example}

For comparison purposes, we employ various CQs
to solve Eq. \eqref{IEex}, including
BDF2CQ, TRCQ,
BGACQ$_{0,1}^{3},$ BGACQ$_{0,2}^{4},$ and BGACQ$_{1,2}^{5},$
$k-$stage Lobatto IIIC RKCQ (LRKCQ$_k$)
and $k-$stage Gauss RKCQ (GRKCQ$_k$) with $k=3,4,5.$
All approaches utilize a  total of 120 quadrature nodes.
Consequently,
the $3-$stage, $4-$stage and
$5-$stage methods  use
$40,$ $30$ and $24$ blocks, respectively.
Computed results are shown in Figure \ref{Ex4F1}.
The first subfigure of Figure \ref{Ex4F1} demonstrates that BDF2CQ suffers from significant departure from the true solution due to its strong dissipative nature. While TRCQ is inherently conservative, its accuracy is also compromised in this case by an insufficiently small time step size.
While BGACQ exhibits a slower convergence rate compared to LRKCQ and GRKCQ with the same stage,
the second and third subfigures in Figure \ref{Ex4F1} reveal that BGACQ's damping errors lie between LRKCQ and GRKCQ,
suggesting GRKCQ's efficiency across the compared schemes.
Notably, with increasing convergence order of CQs (see the fourth subfigure), the damping becomes negligible.
Our numerical experiments have also identified several conservative BGACQs,
such as BGACQ$_{1,1}^{4},$ that achieve numerical accuracy comparable to the conservative GRKCQ$_4.$
Figure \ref{Ex4F2} compares  BGACQ$_{1,1}^{4}$ with GRKCQ$_4,$
demonstrating that the difference in numerical results
computed by these methods is insignificant, while  a significant damping effect is observed for BGACQ$_{0,2}^{4}$ in the third subfigure of Figure \ref{Ex4F1}.
However, analyzing the convergence properties of these conservative BGACQs remains an open challenge for future research.

\begin{figure}
\begin{center}
\includegraphics[width=14cm,height=6.5cm]{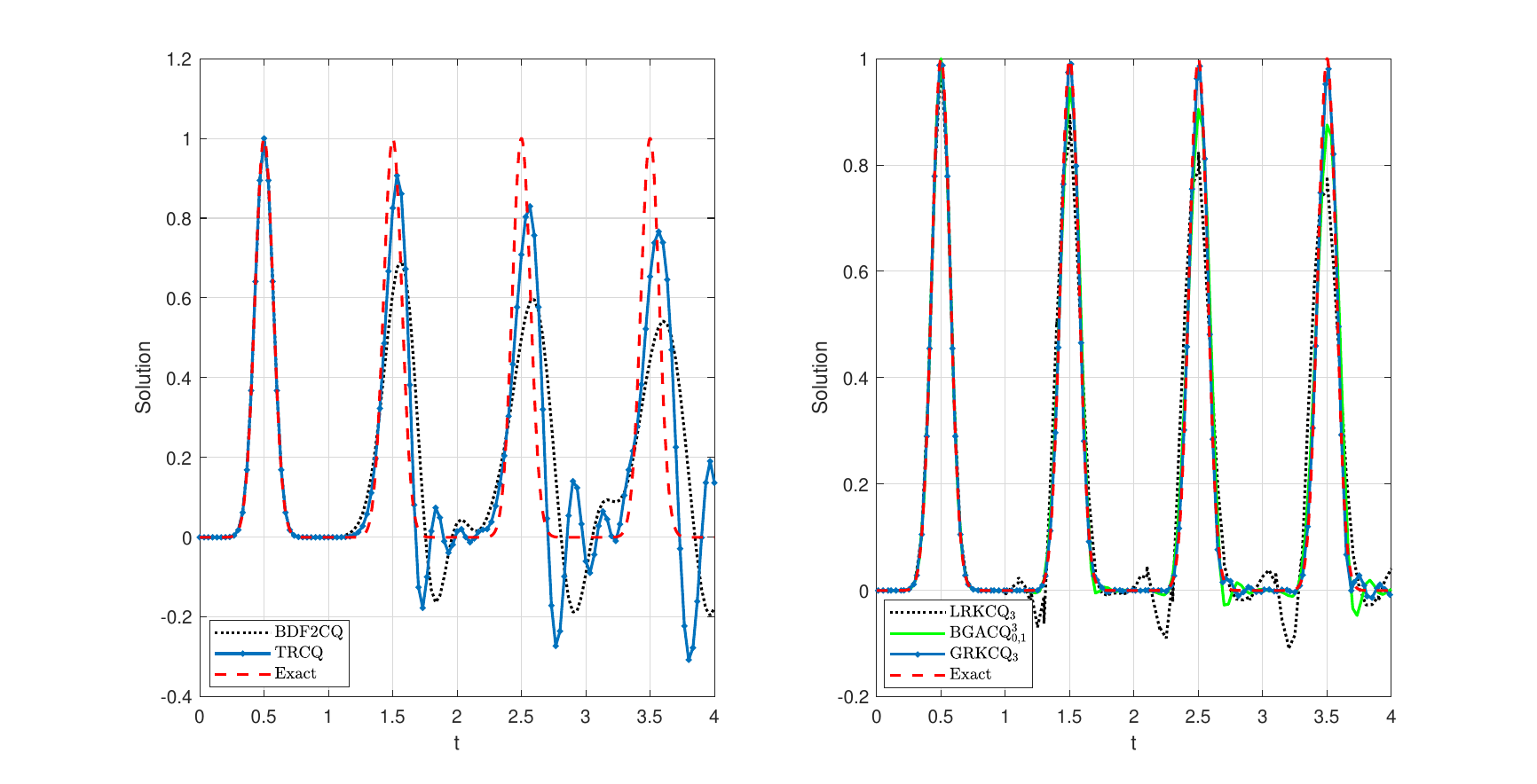}
\includegraphics[width=14cm,height=6.5cm]{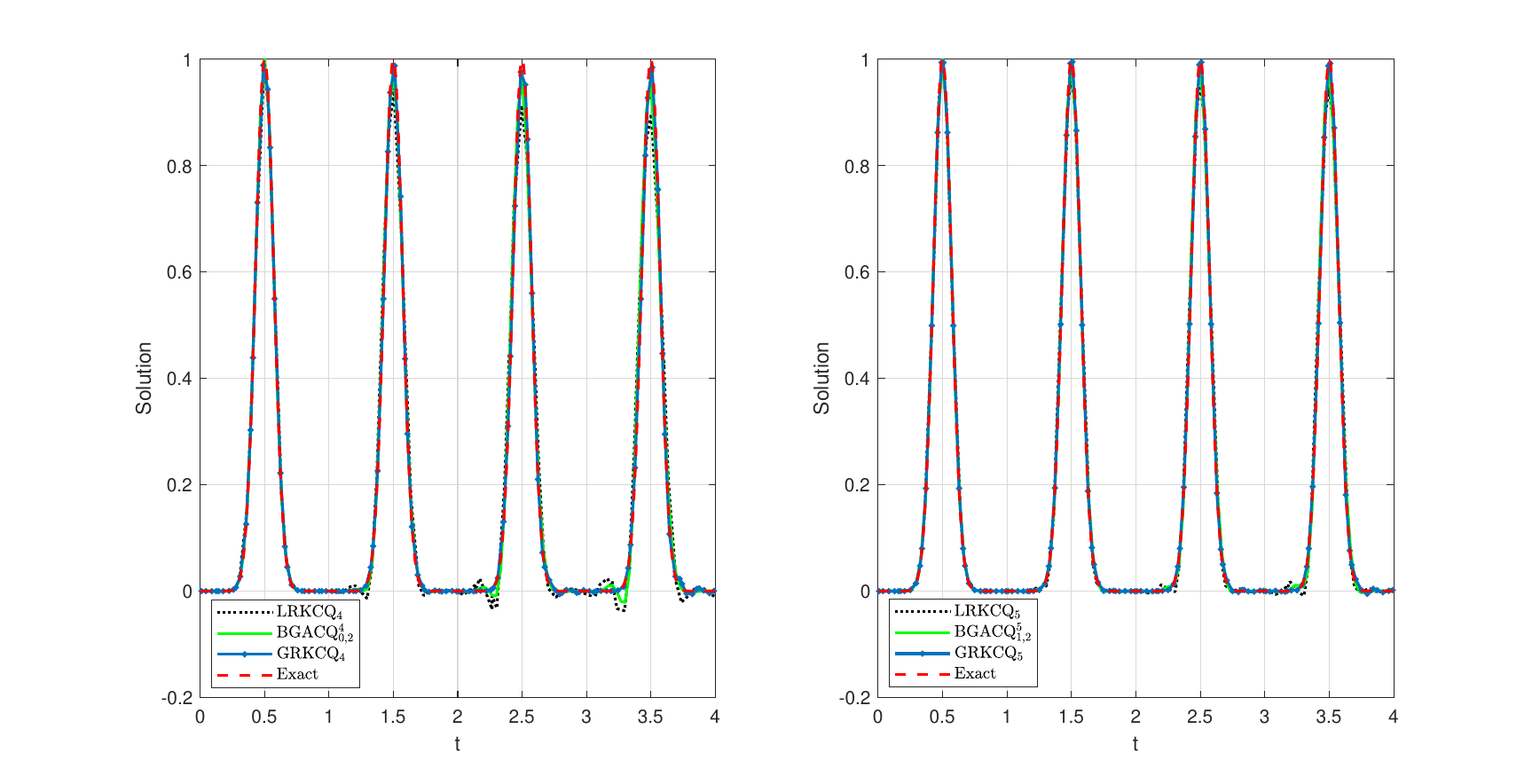}
\caption{Comparison of numerical solutions computed by CQs for Eq. \eqref{IEex} with 120 quadrature nodes.}\label{Ex4F1}
\end{center}
\end{figure}

\begin{figure}
\begin{center}
\includegraphics[width=14cm,height=6.5cm]{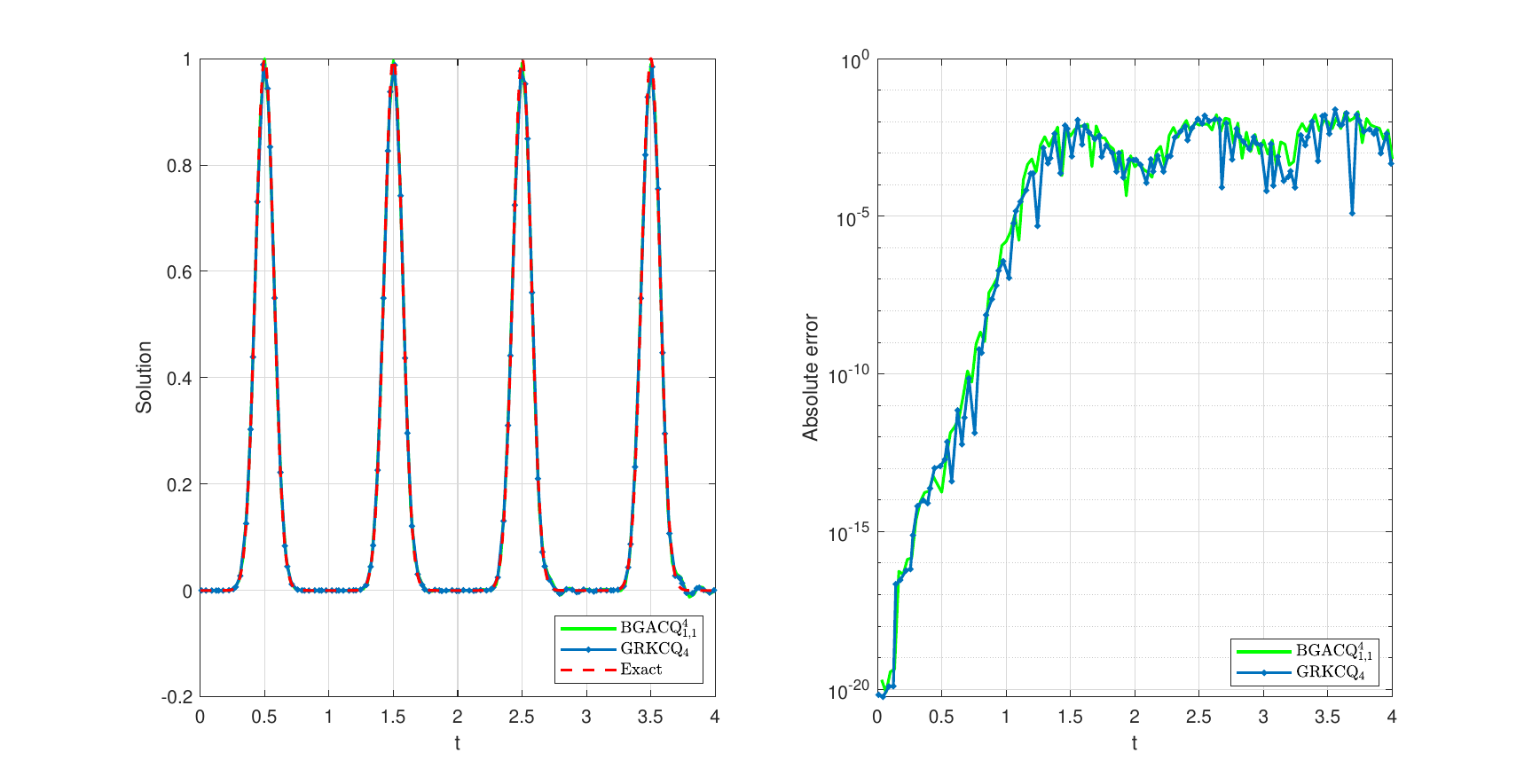}
\caption{Comparison of  numerical solutions computed by  BGACQ$_{1,1}^{4}$ and GRKCQ$_4$ for Eq. \eqref{IEex} with 120 quadrature nodes.}\label{Ex4F2}
\end{center}
\end{figure}

While conservative CQs offer advantages in certain scenarios, dissipative CQs can demonstrate superior performance for specific problems.
Here, we illustrate the benefits of dissipative CQs by applying them to the following problem.

\begin{example}
Consider calculation of the oscillatory integral with a
fast decaying integrand
\begin{align}\label{CI5ex}
\int_0^2J_0(\omega(2-s))\mathrm{e}^{-10s}\frac{s^6}{1+25s^2}ds.
\end{align}
Here $J_0$ denotes the first-kind Bessel function of order $0.$
Its Laplace transform is $(\omega^2+\lambda^2)^{-1/2},$
which is singular at $\lambda=\pm \mathrm{i}\omega.$
As $\omega\rightarrow \infty,$
the problem becomes increasingly hyperbolic.
\end{example}

We employ the $4-$stage methods, GRKCQ$_4,$
BGACQ$_{1,1}^4$ and BGACQ$_{0,2}^4,$
to compute CI \eqref{CI5ex}.
The relative errors, presented in Figure~\ref{Ex5F1}, demonstrate that the dissipative CQ (BGACQ$_{0,2}^4$) consistently outperforms the two conservative methods
regardless of an increase in either the quadrature nodes or the oscillation parameter $\omega.$
This superiority can be attributed, in part, to the faster decay of the dissipative CQ's stability function near the imaginary axis compared to the conservative CQs.
This faster decay allows the dissipative CQ to significantly mitigate rounding errors.

\begin{figure}
\begin{center}
\includegraphics[width=14cm,height=6.5cm]{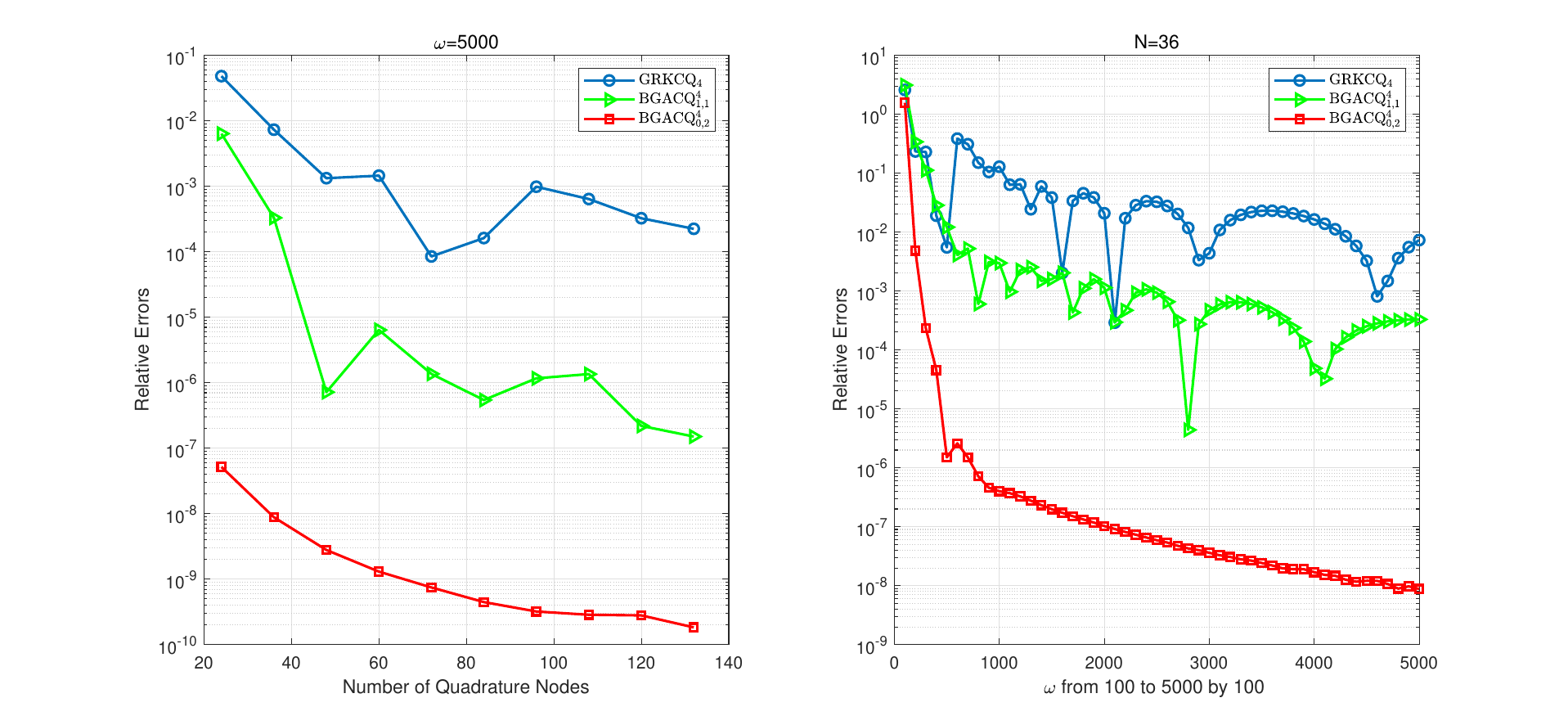}
\caption{Comparison of  dissipative and conservative CQs for CI \eqref{CI5ex}. The left part shows the relative errors as the quadrature nodes increase.
The right part shows the relative errors as $\omega$ increases.}\label{Ex5F1}
\end{center}
\end{figure}

\section{Conclusions}

This paper explores a family of BGACQs for calculation of CI \eqref{CI} with
a hyperbolic kernel.
These quadrature rules can be understood as modifications of  LMCQs or RKCQs.
Notably, BGACQs are implemented on the uniform grid and offer the advantage of adjustable convergence rates without requiring additional grid points.
Furthermore, when coupled with an ODE solver satisfying Assumption~\ref{ASS},
BGACQs demonstrate high-order accuracy and effectiveness in solving a wide range of hyperbolic problems.
Their inherent flexibility suggests significant potential for diverse numerical applications, including fractional differential equations, time-domain boundary integral equations, Volterra integral equations, and potentially many more.

While this paper focuses on BGACQs with small block sizes, higher-order schemes or capturing global information necessitate larger block sizes, leading to a significant increase in computational cost.
Fast  algorithms for calculation of quadrature weights, as explored in \cite{lopez2016generalized}, offer a potential remedy.
This challenge is also crucial for developing a varying-stepsize version of BGACQs.

\section*{Acknowledgements}

We express our sincere gratitude to  anonymous reviewers for their insightful comments, which significantly enhanced the quality of our paper.

\end{document}